\documentclass[a4paper,reqno]{amsart}

\usepackage{amscd,amsthm,amsmath,amssymb,mathtools}
\usepackage{verbatim}
\usepackage{color}
\usepackage{url}
\usepackage{enumerate,enumitem}
\usepackage{graphicx}
\usepackage{epstopdf,epsfig,subfigure}
\usepackage{curve2e}
\usepackage{array}
\usepackage{stackrel}

\usepackage[numbers,sort&compress]{natbib}

\usepackage{hyperref}

\theoremstyle{plain}

\newtheorem{theorem}{Theorem}[section]

\newtheorem{lemma}[theorem]{Lemma}
\newtheorem{corollary}[theorem]{Corollary}
\newtheorem{problem}[theorem]{Problem}
\newtheorem{assumption}[theorem]{Assumption}

\theoremstyle{definition}

\newtheorem{remark}[theorem]{Remark}
\newtheorem{example}[theorem]{Example}

\numberwithin{equation}{section}

\usepackage{geometry}

\newcommand{\rest}{\left.\kern-2\nulldelimiterspace\right|_}
\newcommand{\norm}[2]{\left|#1\right|_{#2}}

\newcommand{\Id}{{\mathbf1}}
\newcommand{\indf}{1}

\newcommand{\p}{\partial}

\newcommand*{\Bigcdot}{\raisebox{-.25ex}{\scalebox{1.25}{$\cdot$}}}

\newcommand{\clC}{{\mathcal C}}

\newcommand{\clE}{{\mathcal E}}

\newcommand{\clG}{{\mathcal G}}
\newcommand{\clH}{{\mathcal H}}

\newcommand{\clJ}{{\mathcal J}}

\newcommand{\clL}{{\mathcal L}}

\newcommand{\clN}{{\mathcal N}}

\newcommand{\clP}{{\mathcal P}}
\newcommand{\clQ}{{\mathcal Q}}
\newcommand{\clR}{{\mathcal R}}

\newcommand{\clX}{{\mathcal X}}
\newcommand{\clY}{{\mathcal Y}}

\newcommand{\bbN}{{\mathbb N}}

\newcommand{\bbR}{{\mathbb R}}

\newcommand{\bfC}{{\mathbf C}}

\newcommand{\bfJ}{{\mathbf J}}

\newcommand{\fkD}{{\mathfrak D}}

\newcommand{\fkI}{{\mathfrak I}}

\newcommand{\fkV}{{\mathfrak V}}

\newcommand{\fkX}{{\mathfrak X}}

\newcommand{\rmD}{{\mathrm D}}

\newcommand{\bfu}{{\mathbf u}}

\newcommand{\bfy}{{\mathbf y}}

\newcommand{\rmd}{{\mathrm d}}
\newcommand{\rme}{{\mathrm e}}
\newcommand{\rmf}{{\mathrm f}}

\definecolor{DarkBlue}{rgb}{0,0.08,0.45}
\definecolor{DarkRed}{rgb}{.65,0,0}
\definecolor{applegreen}{rgb}{0.55, 0.71, 0.0}

\newcounter{mymac@matlab}
\newcommand{\matlab}{MATLAB%
   \ifnum\value{mymac@matlab}<1%
   \textregistered%
   \setcounter{mymac@matlab}{1}%
   \fi%
  }

\providecommand{\argmin}{\operatorname*{argmin}}

\newcommand{\black}{ \color{black} }

\begin{document}
\title{Remarks on finite and infinite time-horizon optimal control problems}
\author{S\'ergio S.~Rodrigues$^{\tt1}$ }
\thanks{
\vspace{-1em}\newline\noindent
\newline\noindent
$^{\tt1}$ Johann Radon Institute for Computational and Applied Mathematics, \"OAW, Altenbergerstrasse~69, 4040~Linz, Austria.
\newline\noindent  
{\sc Email}:
{\small\tt sergio.rodrigues@ricam.oeaw.ac.at}%
}

\begin{abstract}
The minimization of energy-like cost functionals is addressed in the context of optimal control problems. For a general class of dynamical systems, with possibly unstable and nonlinear free dynamics, it is shown that a sequence of solutions of finite time-horizon optimal control problems approximates a solution of  the analog infinite time-horizon  problem. The latter solution and corresponding optimal cost value function are not  assumed to be known a-priori.  Numerical simulations are presented validating the theoretical findings for several examples, including  systems governed by both ordinary and partial differential equations. 
\end{abstract}

\maketitle

\pagestyle{myheadings} \thispagestyle{plain} \markboth{\sc S. S. Rodrigues}
{\sc Remarks on FTH and ITH optimal control}

\maketitle

\section{Introduction}
The goal of this manuscript is the investigation of the behavior of the solutions of finite time-horizon (FTH) control problems as the time-horizon increases, for a general class of dynamical systems, possibly unstable and nonlinear.
The main result
states that the solutions a sequence of FTH optimal control problems, minimizing a given  energy-like cost functional, do approximate a solution of the analog infinite time-horizon (ITH)  problem.
The possibility of such an approximation is important for applications where, in general, the computation of an ITH solution is numerically unfeasible and we have to perform the computations in a FTH. This is done, for example, in stabilization strategies as receding horizon control (RHC) ~\cite{AzmiKun19,KunPfeiffer20,MayneMichalska90,LiYanShi17}, also known as model predictive control  (MPC)~\cite{GruneSchaSchi22,MayneRawlRaoScok00}, where a sequence of open-loop optimal control problems is solved
in a large enough FTH in order to capture the stabilizing property of the IFT optimal control. See also~\cite{GrunePannek09} for discrete time systems.
At this point we should mention that RHC can also be useful for large FTHs, since computations are faster for smaller FTHs and it can increase the robustness of the computed control~\cite{DontchevKolKraVelVuo20}.

Features of solutions of ITH optimal control problems have been investigated in the literature as well as conditions under which such features can be deduced from sequences of analog FTH problems, as the time-horizon diverges to infinity. Concerning~{\sc ode}s, already in~\cite{Halkin74} the optimality conditions for ITH problems have been investigated, and a comparison has been made with appropriate FTH control problems. For more recent works we refer to~\cite[sect.~3]{CannarsaFrankowska18}, \cite{AseevVeliov15,Aseev18,OrlovRov22,AseevBesovKrya12} and references therein. 
In the context of~{\sc pde}s we can mention the recent work~\cite[sect.~3]{CasasKunisch22} where appropriately constructed FTH problems are used to derive a ``limit'' optimality system for an a-priori fixed ITH solution. In particular, the FTH cost functionals in~\cite{CasasKunisch22} depend on the fixed ITH optimal solution. In this paper the goal is different and the FTH cost functionals do not depend on any particular solution of the ITH problem. Essentially, without assuming that an ITH optimal solution is known a-priori, we want to know whether such a solution can be approximated by a sequence of FTH ones. At this point, we must mention that even if we know a solution of the ITH, its approximation is still an important problem, in particular, for infinite-dimensional systems. We refer to~\cite{BanKunisch84,ZwartMorIft20} addressing (different approaches on) the approximation of the solution of  algebraic Riccati equations giving us the optimal feedback control for the case of autonomous linear dynamics and quadratic cost functionals.

The result in this manuscript, concerning a general class of possibly unstable and nonlinear systems, can be related to the result in~\cite[Ch.~III, sect.~6.3]{Lions71}, where a class of stable linear systems is considered.
 
We can also mention that ITH control problems are important in economic sciences. We refer to~\cite[Intro.]{AseevBesovKrya12} where such problems are discussed as applications to economic models/systems which do not propose a natural choice for the planning time-horizon, being assumed to operate for indefinitely long time.

\bigskip\noindent
{\bf Contents and notation.}
We prove the main result in section~\ref{S:optimalcontrol} under general assumptions on our dynamical system. Such assumptions are discussed in section~\ref{S:rmksAssumptions}  as well as the class of cost functionals considered. In section~\ref{S:numerics} we present  numerical examples highlighting several points of the result. These include both illustrative autonomous and nonautonomous {\sc ode} models as well as an autonomous  {\sc pde} model. 
Finally, section~\ref{S:finalremks} gathers concluding remarks.

\medskip
Concerning notation, we write~$\bbR$ and~$\bbN$ for the sets of real numbers and nonnegative
integers, respectively. Then, we denote~$\bbR_s\coloneqq(s,+\infty)$, for~$s\ge0$, and~$\bbN_+\coloneqq\bbN\setminus\{0\}$. The closure of a subset~$I\subseteq\bbR_0$ is denoted~$\overline I$.

Given Banach spaces~$X$ and~$Y$, we write $X\xhookrightarrow{} Y$ if the inclusion
$X\subseteq Y$ holds and is continuous.
The space of continuous functions from~$X$ into~$Y$ is denoted by~$\clC(X,Y)$ and its subspace of linear mappings
by~$\clL(X,Y)$.
The continuous dual of~$X$ is denoted by~$X'\coloneqq\clL(X,\bbR)$ and the
adjoint of an operator $L\in\clL(X,Y)$ by~$L^*\in\clL(Y',X')$.

Given open intervals~$I_j\subset\bbR$, $j\in\{1,2\}$,  with $\sup I_1=\inf I_2$, and spaces~$X_{I_j}$ of functions~$g_j\colon I_j\to X$, we denote concatenated functions by
\begin{equation}\notag
g_1\ddag g_2(t)\coloneqq\begin{cases}
g_1(t),\quad&\mbox{if }t\in I_1;\\
g_2(t),\quad&\mbox{if } t\in I_2;
\end{cases}
\end{equation}
and set~$X_{I_1}\ddag X_{I_2}\coloneqq\{g_1\ddag g_2\mid (g_1,g_2)\in X_{I_1}\times X_{I_2}\}$.

Finally, we introduce the subset
\begin{equation}\notag
\fkI_0\coloneqq\{\fkD\in\clC(\overline\bbR_0,\overline\bbR_0)\mid \fkD\mbox{ is nondecreasing, }\fkD(0)=0\}.
\end{equation}
 
 With~$N\in\bbN_+$, we write~$\norm{\Bigcdot}{\bbR^N}=\norm{\Bigcdot}{\ell^2}$ for the usual Euclidean norm in~$\bbR^N$.
 
 \section{Optimal control problems}\label{S:optimalcontrol}
 We show that a solution of an ITH optimal control problem can be found as the limit of a sequence of solutions of ITH optimal control problems. 

In the following~$\clH $ is a real and separable Hilbert space, and~$I$ stands for an arbitrary open interval
\begin{equation}\notag
I=(\iota_0,\iota_1)\subseteq(0,+\infty).
\end{equation}

\subsection{General setting} 

We consider, for time~$t\in I$, the control system
\begin{subequations}\label{sys-y}
\begin{align} 
 &\dot y(t) =f(t,y(t) ,u(t) ),\quad y(\iota_0)=z\in \clH ,
\intertext{with controls~$u$ subject to pointwise  constraints}
&u(t)\in \bfC\subset\bbR^M,
 \end{align}
 where~$\bfC$ is a closed convex subset of~$\bbR^M$.
\end{subequations}
We consider  cost functionals as
\begin{subequations}\label{Jz}
\begin{align}\label{Jz0}
\hspace{-2em}\clJ_{I}(y,u)&\coloneqq\tfrac12\norm{\clQ(y)}{L^2(I,{\bbR}
)}^2+\tfrac12\norm{\clN(u)}{L^2(I,\bbR)}^2,
\\
\label{JzP}
\hspace{-2em}\clJ_{I}^\clP\!(y,u)\!&\coloneqq\!\!\begin{cases}
\!\clJ_{I}(y,u),&\!\!\iota_1\!=\!+\infty;\\
\!\clJ_{I}(y,u)\!+\!\tfrac12\norm{\clP(y(\iota_1))}{\bbR}^2,&\!\!\iota_1\!<\!+\infty;
\end{cases}\!\!\!\!\!\!
\end{align}
\end{subequations}
where~$\clQ\in\clC(\clH ,\bbR)$, ~$\clN\in\clC(\bbR^M,\bbR)$, and~$\clP\in\clC(\clH ,\bbR)$.

We look for the minimizers~$(y,u)$ of~$\clJ_{I}^\clP$ in a Cartesian separable Banach   space
\begin{align}
&\clX_I\coloneqq \clX_I^y\times \clX_I^u.\notag
\intertext{We introduce another separable Banach space}
&\clY_I\coloneqq \clY_I^{f}\times \clH,\notag
\intertext{and assume that the following mapping is well defined}
\clG_I^{z}\colon\clX_I&\to\clY_I,\notag\\
(y,u)&\mapsto\Bigl(\dot y -f(\Bigcdot,y ,u),\,y(\iota_0)-z\Bigr).\notag
\end{align}

Finally, we introduce the convex set
\begin{align}
&\bfC_I\coloneqq\{(y,u)\in\clX_I\mid u(t)\in\bfC \mbox{ for } t\in I\}\notag
\end{align}
and the subset of~$\clX_I$ as 
\begin{align}
\hspace{-1em}\fkX_{I}^{z}\coloneqq\{(y,u)\in\clX_{I}\cap\bfC_I\mid \clG_I^{z}(y,u)=(0,0)\}.\notag
\end{align}

Our minimization problem is as follows. 
\begin{problem}\label{Pb:OCPI}
Find
$(y_{I}^{z} ,u_{I}^{z})\in\argmin\limits_{(y,u)\in\fkX_{I}^{z}}\clJ_I^P(y,u)$, for given~$z\in \clH $ and time interval~$I$.
\end{problem}

\begin{example}
For finite-dimensional systems as $\dot y=Ay+Bu$, with~$A\in\bbR^{N\times N}$ and~$B\in\bbR^{N\times M}$, $M\le N$, we can take~$\clH =\bbR^N$,  $\clX_I^y=W^{1,2}(I,\bbR^N)\subset\clC(\overline I,\bbR^N)$, and~$\clX_I^u=L^2(I,\bbR^M)$.  
\end{example}
\begin{example}
For infinite-dimensional systems as diffusion-reaction equations $\dot y=\Delta y+ay+Bu$ defined in a spatial bounded domain~$\Omega\subset\bbR^d$ under homogeneous Dirichlet boundary conditions,   $Bu\coloneqq\sum_{j=1}^Mu_j\Phi_j$, with actuators~$\Phi_j\in L^2(\Omega)$, and where weak solutions are well defined,  we can take~$\clH =L^2(\Omega)$,  $\clX_I^y=\{y\in L^2(I,W^{1,2}_0(\Omega))\mid \dot y\in L^2(I,W^{-1,2}(\Omega))\}\subset\clC(\overline I,L^2(\Omega))$, and~$\clX_I^u=L^2(I,\bbR^M)$.  
\end{example}

Hereafter, in sentences as ``$(y_{I}^{z} ,u_{I}^{z})$ solves Problem~\ref{Pb:OCPI}'' it is understood that we are referring to the initial state~$z$ and interval~$I$, included in the notation.

We denote the ITH  value function~$\fkV_s\colon\clH \to\bbR$ as
\begin{equation}\notag
 \fkV_s(z)\coloneqq\clJ_{\bbR_{s}}(y_{\bbR_{s}}^{z} ,u_{\bbR_{s}}^{z} ),
\end{equation}
with~$(y_{\bbR_{s}}^{z} ,u_{\bbR_{s}}^{z} )$ solving Problem~\ref{Pb:OCPI}.

We make the following assumptions.

\begin{assumption}\label{A:cont}
For every~$s\ge0$, the dynamical system~\eqref{sys-y} is well posed locally in time for data~$(z,u)\in\clH \times \clX_I^u$ with~$u(t)\in\bfC$. There exist
$t_0=t_0(z,u)>0$ and~$\fkD\in\fkI_0$, so that for all~$\tau\in(0,t_0)$
\begin{equation}\notag
\norm{y}{\clX_{(s,s+\tau)}^y}\le \fkD(\tau)(\norm{(z,u)}{\clH \times \clX_{(s,s+\tau)}^u}).
\end{equation}
 \end{assumption}
 
 \begin{assumption}\label{A:conc}
We have $\clX_{I_1}^y\subseteq \clC(\overline I_1,\clH )$, for every interval~$I_1\subseteq\bbR_0$.
Further, for every interval~$I_2\subset\bbR_0$ with $\sup I_1=\inf I_2$, we have $\clX_{(\inf I_1,\sup I_2)}^u=\clX_{I_1}^u\ddag \clX_{I_2}^u$
and~$\clX_{(\inf I_1,\sup I_2)}^y=(\clX_{I_1}^y\ddag\clX_{I_2}^y)\cap\clC(\overline I_1\cup \overline I_2,\clH )$.
 \end{assumption}

 \begin{assumption}\label{A:Qweakc}
The maps $\widehat \clQ\colon \clX_{\bbR_s}^y\to L^2(\bbR_s,\bbR)$ and $\widehat \clN\colon \clX_{\bbR_s}^u\to L^2(\bbR_s,\bbR)$, defined as $(\widehat \clQ(y))(t)\coloneqq \clQ(y(t))$ and $(\widehat \clN(u))(t)\coloneqq \clN(u(t))$, are weakly continuous.
 \end{assumption}
 \begin{assumption}\label{A:weakcont}
For every $c>0$ and interval~$I\subseteq\bbR_{0}$, the set $\{(y,u)\in \fkX_{I}^{z}\mid \clJ_{I}(y ,u )\le c\}$ is weakly compact.
 \end{assumption}
\begin{assumption}\label{A:exist-optim}
For every~$s\ge0$, with~$I=\bbR_{s}$ there exists at least one solution for Problem~\ref{Pb:OCPI} and 
~$\fkV_s(z)\le \tfrac12\norm{\clR(z)}{\bbR}^2$,
for a function~$\clR\in\clC(\clH ,\bbR)$ independent of~$s$, with~$\clR(0)=0$.
\end{assumption}
  \begin{assumption}\label{A:P}
For all~$z\in\clH $, we have $\norm{\clP(z)}{\bbR^2}^2\le\fkD_1(\fkV_s(z))$, for a function~$\fkD_1\in\fkI_0$ independent of~$s$.
 \end{assumption}
 \begin{assumption}\label{A:R}
For each~$s\ge0$ and $z\in\clH $, it holds the following. For every~$r>s$ there exists a sequence~$({T}_n,{T}^\circ_n)_{n\in\bbN_+}$, such that~$T_n\to+\infty$,
\begin{equation}\notag
r\le{T}^\circ_n\le{T}_n<+\infty, \mbox{ and }\lim\limits_{n\to+\infty}\!\fkV_{T^\circ_n}(y_{(s,T_n)}^{z}(T^\circ_n))=0,
\end{equation}
where~$(y_{(s,T_n)}^{z} ,u_{(s,T_n)}^{z})$ solves Problem~\ref{Pb:OCPI}.
 \end{assumption}

 Next, we recall the dynamic programming principle.
 \begin{lemma}\label{L:DPP}
For every~$T>s$, with~$I^T\coloneqq(s,T)$,
\begin{align}\notag
\hspace{-2em}&(\overline y ,\overline u)\in\argmin\limits_{(y,u)\in\fkX_{\bbR_s}^{z}}\clJ_{\bbR_s}(y,u)\\
\hspace{-2em}\Longleftrightarrow\quad&(\overline y ,\overline u)\rest{I^T}\in\argmin\limits_{(w,v)\in\fkX_{I^T}^{z}}\left(\clJ_{I^T}(w,v)+\fkV_T(w(T))\right).\notag
\end{align}
\end{lemma}
\begin{proof}
Let us be given~$(\overline y ,\overline u)\in\fkX_{\bbR_s}^{z}$ minimizing~$\clJ_{\bbR_s}$ and~$(\overline w ,\overline v)\in\fkX_{I^T}^{z}$ minimizing~$\clJ_{I^T}(w,v)+\fkV_T(w(T))$, then by Assumption~\ref{A:conc} and optimality we find
\begin{align}
&\clJ_{\bbR_s}(\overline y ,\overline u)\le\clJ_{I^T}(\overline w ,\overline v)+\clJ_{\bbR_T}( y_{\bbR_T}^{\overline w(T)} ,u_{\bbR_T}^{\overline w(T)})\notag\\
&\le\clJ_{I^T}(\overline y ,\overline u)+\clJ_{\bbR_T}( y_{\bbR_T}^{\overline y(T)} ,u_{\bbR_T}^{\overline y(T)})\notag\\
&\le\clJ_{I^T}(\overline y ,\overline u)+\clJ_{\bbR_T}(\overline y ,\overline u)=\clJ_{\bbR_s}(\overline y ,\overline u),\label{dpp=}
\end{align}
which finishes the proof.
\end{proof}

\begin{corollary}\label{C:DPP}
For~$(y_{\bbR_s}^{z} ,u_{\bbR_s}^{z})$ solving Problem~\ref{Pb:OCPI}, 
\begin{equation}\label{dpp-cor}
\lim_{T\to+\infty}\fkV_T(y_{\bbR_s}^{z}(T))=0\quad\mbox{if}\quad \fkV_s(z)<+\infty.
\end{equation}
 \end{corollary}
 \begin{proof}
 By Lemma~\ref{L:DPP} and~\eqref{dpp=} in its proof,  we have $\fkV_T(y_{\bbR_s}^{z}(T))=\clJ_{\bbR_T}( y_{\bbR_s}^{z} ,u_{\bbR_s}^{z})=\fkV_s(z)-\clJ_{I^T}(y_{\bbR_s}^{z} ,u_{\bbR_s}^{z})$. Due to the integral form of the cost functional~\eqref{Jz0}, we obtain~$\fkV_T(y_{\bbR_s}^{z}(T))\to0$ as~$T\to+\infty$.
 \end{proof}
The main result of this manuscript is as follows.

\begin{theorem}\label{T:limitOCP}
Let Assumptions~\ref{A:cont}--\ref{A:R} hold true.
Let~$z\in\clH $ and ${T}>s\ge0$. Then, there exists a strictly increasing divergent sequence~$({T}_n)_{n\in\bbN}$, with~$s<{T}_n$, as in Assumption~\ref{A:R}, such that
\begin{equation}\notag
\lim_{n\to+\infty}\clJ_{(s,{T}_n)}^P(y_{(s,{T}_n)}^{z} ,u_{(s,{T}_n)}^{z})=\clJ_{\bbR_{s}}^P(y_{\bbR_{s}}^{z} ,u_{\bbR_{s}}^{z}),
\end{equation}
where~$(y_{(s,{T})}^{z},u_{(s,{T})}^{z})$ and~$(y_{\bbR_{s}}^{z},u_{\bbR_{s}}^{z})$ solve Problem~\ref{Pb:OCPI}.
Furthermore, given an arbitrary~$r>s$ 
there exists a subsequence~${T}_{\sigma(n)}$, with~${T}_{\sigma(n)}\ge r$, so that the restrictions of the optimal pairs to $(s,r)$ converge weakly,
\begin{equation}\label{weakconv-sr}
(y_{(s,{T}_{\sigma(n)})}^{z},u_{(s,{T}_{\sigma(n)})}^{z})\rest{(s,r)}
\xrightharpoonup[\clX_{(s,r)}]{}
(\widehat y,\widehat u)\rest{(s,r)},
\end{equation}
where~$(\widehat y,\widehat u)\in\fkX_{\bbR_s}^z$ solves Problem~\ref{Pb:OCPI}.
\end{theorem}

\begin{proof}
We fix~$s\ge0$. For an arbitrarily given ${T}> s$, denote the time interval~$I^{T}\coloneqq(s,{T})$ and let~$(y_{I^{T}}^{z},u_{I^{T}}^{z})$ solve Problem~\ref{Pb:OCPI}.  
Then, by optimality,
\begin{align}
&\clJ_{I^{T}}^\clP(y_{I^{T}}^{z},u_{I^{T}}^{z})
\le \clJ_{I^{T}}^\clP(y_{\bbR_{s}}^{z}\rest{I^{T}},u_{\bbR_{s}}^{z}\rest{I^{T}})\notag\\
&\hspace{0em}= \clJ_{I^{T}}(y_{\bbR_{s}}^{z}\rest{I^{T}},u_{\bbR_{s}}^{z}\rest{I^{T}})+\tfrac12\norm{\clP(y_{\bbR_{s}}^{z}({T}))}{\bbR}^2
\notag\\
&\hspace{0em}\le \fkV_s(z)+\tfrac12\norm{\clP(y_{\bbR_{s}}^{z}({T}))}{\bbR}^2\label{opt1}
\end{align}
and, by Assumption~\ref{A:P}, 
\begin{align}
&\hspace{0em}\clJ_{I^{T}}^\clP(y_{I^{T}}^{z},u_{I^{T}}^{z})
\le \fkV_s(z)+\tfrac12\fkD_1(\fkV_{T_n}(y_{\bbR_{s}}^{z}(T))),\notag
\end{align}
which implies that
\begin{align}
\limsup_{T\to+\infty}\clJ_{I^{{T}}}^\clP(y_{I^{{T}}}^{z},u_{I^{{T}}}^{z})&
\le \fkV_s(z),\label{fin<inf}
\end{align}
because, due to~\eqref{dpp-cor} and the finite cost condition in Assumption~\ref{A:exist-optim},
\begin{align}
&\lim_{T\to+\infty}\tfrac12\fkD_1(\fkV_{T}(y_{\bbR_{s}}^{z}(T)))=0.\label{limPR=0P}
\end{align}

Let us fix an arbitrary~$r>s$ and a sequence $({T}_n,{T}^\circ_n)_{n\in\bbN_+}$ as in Assumption~\ref{A:R}: ${T}_n\to+\infty$ and
\begin{align}
&\!\lim_{n\to+\infty}\!\fkV_{T_n^\circ}(y_{I^{T_n}}^{z}(T_n^\circ))=0,\quad r<{T}^\circ_n\le{T}_n.\label{limPR=0R}
\end{align}
Take the restriction~$(y_{I^{{T}_n}}^{z},u_{I^{{T}_n}}^{z})\rest{(s,{T}^\circ_n)}$ of~$(y_{I^{{T}_n}}^{z},u_{I^{{T}_n}}^{z})$ to~$I^{T_n^\circ}=(s,{T}^\circ_n)$, the state~$y_{I^{T}}^{z}({T}^\circ_n)$, and set
\begin{align}
&\hspace{-2em}(\bfy_{{T}^\circ_n},\bfu_{{T}^\circ_n})\notag\\
&\hspace{-2em}=(y_{I^{{T}_n}}^{z},u_{I^{{T}_n}}^{z})\rest{(s,{T}^\circ_n)}\ddag\; (y_{\bbR_{{T}^\circ_n}}^{y_{I^{{T}_n}}^{z}({T}^\circ_n)},u_{\bbR_{{T}^\circ_n}}^{y_{I^{{T}_n}}^{z}({T}^\circ_n)}),\label{Th:concat}
\end{align}
where~$(y_{\bbR_{{T}^\circ_n}}^{y_{I^{{T}_n}}^{z}({T}^\circ_n)},u_{\bbR_{{T}^\circ_n}}^{y_{I^{{T}_n}}^{z}({T}^\circ_n)})$ solves Problem~\ref{Pb:OCPI}.
By Assumption~\ref{A:conc}, $(\bfy_{{T}^\circ_n},\bfu_{{T}^\circ_n})\in\fkX_{I^{{T}^\circ_n}}^{z}\ddag\fkX_{\bbR_{{T}^\circ_n}}^{y_{I^{{T}_n}}^{z}({T}^\circ_n)}=\fkX_{\bbR_s}^{z}$. Thus,  by optimality and Assumption~\ref{A:exist-optim},
\begin{align}
&\fkV_s(z)=\clJ_{{\bbR_{s}}}(y_{\bbR_{s}}^{z},u_{\bbR_{s}}^{z})\le \clJ_{\bbR_{s}}(\bfy_{{T}^\circ_n},\bfu_{{T}^\circ_n})\notag\\
&=\clJ_{I^{T^\circ_n}}(\bfy_{{T}^\circ_n},\bfu_{{T}^\circ_n})+\fkV_{T_n^\circ}(y_{I^{T_n}}^{z}(T_n^\circ))
.\notag\\
&\le\clJ_{I^{{T}_n}}(y_{I^{{T}_n}}^{z},u_{I^{{T}_n}}^{z})+\fkV_{T_n^\circ}(y_{I^{T_n}}^{z}(T_n^\circ))\notag\\
&\le\clJ_{I^{{T}_n}}^\clP(y_{I^{{T}_n}}^{z},u_{I^{{T}_n}}^{z})+\fkV_{T_n^\circ}(y_{I^{T_n}}^{z}(T_n^\circ)),\label{opt2}
\end{align}
which, together with~\eqref{limPR=0R}, lead us to
\begin{align}
\fkV_s(z)&\le 
\liminf_{n\to+\infty}\clJ_{I^{{T}_n}}^\clP(y_{I^{{T}_n}}^{z},u_{I^{{T}_n}}^{z}).\label{inf<fin}
\end{align}

Combining~\eqref{fin<inf} and~\eqref{inf<fin} we obtain
\begin{align}
\lim_{n\to+\infty}\clJ_{I^{{T}_n}}^\clP(y_{I^{{T}_n}}^{z},u_{I^{{T}_n}}^{z})=\fkV_s(z).\label{inf=fin}
\end{align}

Next, from~\eqref{opt2}, \eqref{opt1}, \eqref{limPR=0P}, and~\eqref{limPR=0R}, we have
\begin{align}
&\hspace{-1.5em}\clJ_{\bbR_{s}}(\bfy_{{T}^\circ_{n}},\bfu_{{T}^\circ_{n}})\le\clJ_{I^{{T}_n}}^\clP(y_{I^{{T}_n}}^{z},u_{I^{{T}_n}}^{z})+\fkV_{T_n^\circ}(y_{I^{T_n}}^{z}(T_n^\circ))\notag\\
&\hspace{-1.5em}\le\fkV_s(z)+\tfrac12\norm{\clP(y_{\bbR_{s}}^{z}({T}_{n}))}{\bbR}^2
+\tfrac12\fkV_{T_{n}^\circ}(y_{I^{T_{n}}}^{z}(T_{n}^\circ))\notag\\
&\hspace{-1.5em}\le c,\label{opt4}
\end{align}
with~$c>0$ independent of~$n$. Then, by Assumption~\ref{A:weakcont}, we can take a subsequence~$({T}_{\sigma(n)},{T}^\circ_{\sigma(n)})_{n\in\bbN_+}$ of~$({T}_n,{T}^\circ_n)_{n\in\bbN_+}$ such that
\begin{align}
\!\!\! (\overline\bfy_{n},\overline\bfu_{n})\coloneqq(\bfy_{{T}^\circ_{\sigma(n)}},\bfu_{{T}^\circ_{\sigma(n)}})\xrightharpoonup[\;\clX_{\bbR_{s}}\;]{}
 (\overline\bfy_{\infty},\overline\bfu_{\infty})\label{wconvRs}
\end{align}
holds for a suitable~$(\overline\bfy_{\infty},\overline\bfu_{\infty})\in\fkX_{\bbR_{s}}^z\subset\clX_{\bbR_{s}}$.
By Assumption~\ref{A:Qweakc}, combined with~\eqref{opt4}, \eqref{limPR=0P}, and~\eqref{limPR=0R}, it follows 
$
\clJ_{{\bbR_{s}}}(\overline\bfy_{\infty},\overline\bfu_{\infty})\le\fkV_s({z})
$
and, by  optimality,
\begin{align}
\clJ_{\bbR_{s}}(\overline\bfy_{\infty},\overline\bfu_{\infty})=\fkV_s({z}).\notag
\end{align}
That is,~$(\overline\bfy_{\infty},\overline\bfu_{\infty})$ is a minimizer for~$\clJ_{\bbR_{s}}$ in~$\fkX_{\bbR_{s}}$.
Finally, note that
$(\overline\bfy_{n},\overline\bfu_{n})\rest{(s,r)}=(\bfy_{{T}^\circ_{\sigma(n)}},\bfu_{{T}^\circ_{\sigma(n)}})\rest{(s,r)}=(y_{(s,{T}_{\sigma(n)})}^{z},u_{(s,{T}_{\sigma(n)})}^{z})\rest{(s,r)}$ since,
recalling~\eqref{limPR=0R}, ${T}^\circ_{n}>r$. 

Therefore, by~\eqref{wconvRs} and Assumption~\ref{A:weakcont}, we find
\begin{align}
(y_{(s,{T}_{\sigma(n)})}^{z},u_{(s,{T}_{\sigma(n)})}^{z})\rest{(s,r)}&= (\overline\bfy_{n},\overline\bfu_{n})\rest{(s,r)}\notag\\
&\xrightharpoonup[\;\clX_{(s,r)}\;]{}
 (\overline\bfy_{\infty},\overline\bfu_{\infty})\rest{(s,r)}.\notag
\end{align}
Thus, the theorem holds with~$ (\widehat y,\widehat u) =(\overline\bfy_{\infty},\overline\bfu_{\infty})$.
\end{proof}

\section{Remarks on the assumptions}\label{S:rmksAssumptions}
All assumptions but Assumptions~\ref{A:P} and~\ref{A:R}, concern either the well-posedness of the dynamical system or properties that we will likely need for the proof of existence of optimal controls (e.g., as the limit of a minimizing sequence). Hence we focus the discussion on Assumptions~\ref{A:P} and~\ref{A:R}.

\subsection{The functions~$\clQ$, $\clP$, and~$\clR$}
In the literature,~$\clQ$ is often taken as a norm (e.g., $\clQ(y(t))=\norm{y(t)}{\clH }$), sometimes just for simplicity. However, in applications it may be convenient to take~$\clQ$ with nontrivial zeros,~$\clQ^{-1}(\{0\})\ne\{0\}$. An example is the computation of optimal stabilizing controls for linear systems with a large number~$D$ of degrees of freedom (e.g., given by discrete accurate spatial approximations of autonomous parabolic equations). In this case, if~$\clQ\in\bbR^{n\times D}$ has low rank~$n\ll D$ (say, $n$ is small when compared to~$D$) the computation of the solution~$\Pi\in\bbR^{D\times D}$ of the algebraic Riccati equation associated with the optimal control (for quadratic cost functionals) can become an easier task, in particular, when the number~$M\ll D$ of actuators is also small, we can often look for approximations~$\Pi_\rmf^\top\Pi_\rmf\approx\Pi$, where~$\Pi_\rmf\in\bbR^{r\times D}$ has low rank~$r\ll D$ (e.g., see the discussion in~\cite[sects.~4 and~5]{BennerLiPenzl08}).

\medskip
The choice of~$\clP$, for the FTH problems, will usually be made depending on a desired smallness we would like to achieve/enhance for the state at finite time. We take~$\clP(\Bigcdot)=\norm{\Bigcdot}{\clH }$ if we wish the norm of the final state to be small. Assumption~\ref{A:P} requires that ~$\clP$ is bounded by the optimal cost. This is natural, and the assumption is made essentially to avoid  some particular settings. For example, assume that
we have a system as
 \begin{align}
\dot y&= y,\;& y(0)&=y_0\in V_1,\notag\\
\dot w&=w+Bu\;& w(0)&=w_0\in V_2.\notag
\end{align}
The system is in fact uncoupled. But, if we were uncertain of the modeling and of the dynamics of the component~$y$, we could assume that it was depending on~$w$.
If we would try to enhance the smallness of the norm of the entire state~$(y,w)\in\clH =V_1\times V_2$ at final time, through the minimization of a functional as
\begin{equation}\notag
\tfrac12\norm{w}{L^2((0,T),V_2)}^2+\tfrac12\norm{u}{L^2((0,T),\bbR^M)}^2+\tfrac12\norm{(y,w)}{\clH }^2,
\end{equation}
we can see that, if~$w_0=0$ and~$y_0\ne0$, then the optimal pairs~$(y,w,u)$ satisfy~$(w,v)=0$ and~$y(t)=\rme^{t}y_0$ for both the ITH and FTH problems. However, the optimal costs are different, the ITH one vanishes, while the FTH ones are given by~$\tfrac12\rme^{2T}\norm{(y_0,0)}{\clH }^2$. Note that in this case the ITH value function is independent of~$y$, $\fkV_s(y,w)=\fkV_s(0,w)$, and Assumption~\ref{A:P} does not hold for the chosen~$\clP(y,w)=\norm{(y,w)}{\clH }$.

\subsection{The (in)stability of the dynamics}
Let us consider the linear system
\begin{align}\label{ex1.lin}
\dot y= Ay+Bu,\qquad y(0)=z\in\bbR^N,
\end{align}
with~$A\in\bbR^{N\times N}$, ~$B\in\bbR^{N\times M}$, state~$y(t)\in\bbR^N$, and control~$u(t)\in\bbR^M$, ~$M\le N$.

For simplicity, we set~$\bfC=\bbR$,~$\clP=0$, ~$\clN(u)=\norm{u}{\ell^2}$, and~$\clQ(y)=\norm{\clQ y}{\ell^2}$ with~$\clQ\in\bbR^{q\times N}$, $q\le N$. Hence
\begin{equation}\notag
\clJ_{I}(y,u)=\tfrac12\norm{\clQ y}{L^2(I,\bbR^q)}^2+\tfrac12\norm{u}{L^2(I,\bbR^M)}^2.
\end{equation}

For arbitrary~$T>0$, with~$I_T=(0,T)$, following the proof of Theorem~\ref{T:limitOCP}, with~$s=0$, we can find
\begin{align}
&\clJ_{I^{T}}(y_{I^{T}}^{z},u_{I^{T}}^{z})
\le \fkV_0(z),\label{opt1ex1}
\end{align}
and then the inequality limit~\eqref{fin<inf}. Within the proof, we have subsequently
used Assumption~\ref{A:R} in order to show the reverse inequality limit~\eqref{inf<fin}, and finally derive the equality~\eqref{inf=fin}. We can see that if we were able to construct any other sequence~$(\widetilde y_{I^{{T}}}^{z},\widetilde u_{I^{{T}}}^{z})$, defined in~$\bbR_{0}$, from the FTH sequences~$(y_{I^{{T}}}^{z},u_{I^{{T}}}^{z})$, such that we would still have the inequality
\begin{align}
\clJ_{\bbR_0}(\widetilde y_{I^{{T}}}^{z},\widetilde u_{I^{{T}}}^{z})&
\le \clJ_{{\bbR_{0}}}(y_{\bbR_{0}}^{z},u_{\bbR_{0}}^{z}),\notag
\end{align}
and such that we could take a  subsequence weakly converging to a limit in~$\fkX_{\bbR_0}^z\subset\clC(\bbR_0,\bbR^N)\times L^2(\bbR_0,\bbR^M)$, then we could use the optimality of~$(y_{\bbR_{0}}^{z},u_{\bbR_{0}}^{z})$ to conclude that the limit is a minimizer and, in particular, the convergence of the optimal costs. This is, in fact,  the approach in~\cite[Ch.~III, sect.~6.3]{Lions71}, where the functions~$( y_{I^{{T}}}^{z}, u_{I^{{T}}}^{z})$ are simply extended by zero as
\begin{equation}\label{extby0}
(\widetilde y_{I^{{T}}}^{z}(t),\widetilde u_{I^{{T}}}^{z}(t))\coloneqq\begin{cases}
(y_{I^{{T}}}^{z}(t),u_{I^{{T}}}^{z}(t)),&\;\mbox{if }t\le T;\\
(0,0),&\;\mbox{if }t>T.
\end{cases}
\end{equation}
Note, however, that such sequence~$(\widetilde y_{I^{{T}}}^{z}(t),\widetilde u_{I^{{T}}}^{z}(t))$ is not in~$\fkX_{\bbR_0}^z$ because~$\widetilde y_{I^{{T}}}^{z}(t)$ is not necessarily continuous at~$t=T$. In the case~$A$ is stable, to overcome this issue,   in~\cite[Ch.~III, sect.~6.3]{Lions71},  the stability of the free dynamics is used to derive that~$y_{I^{{T}}}^{z}(T)\le C$ for some constant~$C$ independent of~$T$, which allow us to place the dynamics equation in a more general distributional setting and take the limit, leading to the convergence result; see~\cite[Ch.~III, sect.~6.3]{Lions71} for more details.

For general unstable systems, it is not clear whether we will necessarily have such a uniform boundedness for~$y_{I^{{T}}}^{z}(T)$. To illustrate this point, from~\eqref{ex1.lin} we obtain 
\begin{align}\notag
y(t)= \rme^{A t}z+{\textstyle\int_0^t}\rme^{A (t-\tau)} Bu(\tau)\,\rmd \tau
\end{align}
and, by~\eqref{opt1ex1}, we have that
\begin{align}
\tfrac12\norm{\clQ y_{I^{{T}}}^{z}}{L^2(I^{{T}},\bbR^N)}^2+\tfrac12\norm{u_{I^{{T}}}^{z}}{L^2(I^{{T}},\bbR^M)}^2
\le C_0\notag
\end{align}
with~$C_0\coloneqq\fkV_0(z)$ independent of~$T$. The uniform bound for the control gives us a uniform bound as
\begin{align}\notag
y_{I^{{T}}}^{z}(T)= \rme^{A T}z+{\textstyle\int_0^T}\rme^{A (T-\tau)} Bu_{I^{{T}}}^{z}(\tau)\,\rmd \tau\le C,
\end{align}
if $A$ is stable. But, if~$A$ is unstable we cannot guarantee such a uniform bound from a bound for the control alone (e.g., for~$u=0$ and a suitable~$z$, $y(T)$ will increase exponentially as~$T\to+\infty$).

We  have circumvented the need of such uniform bound by making Assumption~\ref{A:R}, which, moreover, allow us to remain in the functional space setting, by constructing appropriate extensions~\eqref{Th:concat}, in~$\fkX_{\bbR_0}^z$, of the FTH solutions, instead of the extensions as~\eqref{extby0}.

\subsection{The necessity of Assumption~\ref{A:R}}\label{sS:necAR}
We start with a particular case where Assumption~\ref{A:R} is a consequence of the other assumptions. For this purpose, let Assumption~\ref{A:exist-optim} hold true with $\norm{\clR(y)}{\bbR}^2\le\fkD_2(\norm{\clQ(y)}{\bbR}^2)$, where~$\fkD_2\in\fkI_0$. Then Assumption~\ref{A:R} also holds true. Indeed, for every~$T>s$,
with
\begin{align}
\vartheta&\coloneqq\min\left\{\left.\norm{\clQ(y_{I^{T}}^{z}(t))}{\bbR}^2\right| \tfrac{s+{T}}2\le t\le {T}\right\}\notag
\end{align}
we find that
\begin{align}
&\tfrac{{T}-s}2\vartheta\le\norm{\clQ(y_{I^{T}}^{z})}{L^2((\tfrac{s+{T}}2, {T}),\bbR)}^2\le 2\clJ_{I^{T}}(y_{I^{T}}^{z},u_{I^{T}}^{z})\notag
\end{align}
and by
$\clJ_{I^{{T}}}(y_{I^{T}}^{z},u_{I^{T}}^{z})\le\clJ_{I^{T}}^\clP(y_{I^{T}}^{z},u_{I^{T}}^{z})$, 
\begin{align}
&\vartheta
\le\tfrac4{{{T}}-s}\clJ_{I^{{T}_n}}^\clP(y_{I^{{T}}}^{z},u_{I^{{T}}}^{z}).\notag
\end{align}
Recalling~\eqref{opt1} and Assumption~\ref{A:P},
\begin{align}
&\vartheta
\le\tfrac4{{{T}}-s}\left(\fkV_s(z)+\tfrac12\fkD_1(\fkV_{T}(y_{\bbR_{s}}^{z}({{T}})))\right)\eqqcolon\Phi_T.\notag
\end{align}
Thus, choosing~${T}^\circ$ as
\begin{align}
&{T}^\circ\coloneqq\max\Bigl\{t\in[\tfrac{s+{{T}}}2, {{T}}]\left|\;\,\norm{\clQ(y_{I^{{T}}}^{z}(t))}{\bbR}^2
\le \Phi_T\right.\Bigr\},\notag
\end{align}
we obtain,~$\clQ(y_{I^{{T}}}^{z}({T}^\circ))\to0$, hence~$\clR(y_{I^{{T}}}^{z}({T}^\circ))\to0$ and, by Assumption~\ref{A:exist-optim},  ~$\fkV_{{T}^\circ}(y_{I^{{T}}}^{z}({T}^\circ))\to0$. We can conclude that Assumption~\ref{A:R} holds true. Note that, for a given~$r>0$, we choose~$T$ such that~$\tfrac{s+{{T}}}2>r$.

\medskip
Next we show that Assumption~\ref{A:R} is necessary for a large class of systems.
Let us assume, for simplicity, that our dynamics is autonomous so that the IFT value function~$\fkV=\fkV_s$ is independent of the initial time~$s$. By the finite cost condition in Assumption~\ref{A:exist-optim} the ``cost to go'' satisfies, see~\eqref{dpp-cor}, $\lim\limits_{t\to+\infty}\fkV(y_{\bbR_0}^{z}(t))=0$.
Now, if  Assumption~\ref{A:R} does not hold, then there are $z\in\clH $, $r>0$, and~$\varepsilon>0$, such that for all
$(T,t)\in\bbR_r\times[r,T]$ we have that~$\fkV(y_{(0,T)}^{z}(t))\ge\varepsilon$. Resuming, we have
\begin{equation}\label{cost2goFI}
\lim_{t\to+\infty}\fkV(y_{\bbR_0}^{z}(t))=0;\quad \inf_{t\in[r,T]}\fkV(y_{(0,T)}^{z}(t))\ge\varepsilon.
\end{equation}
Let us fix~$T_0\ge r$ so that~$\fkV(y_{\bbR_0}^{z}(t))\le\tfrac12\varepsilon$ for all~$t\ge T_0$.

If the value function is continuous, we have
\begin{align}\label{contV}
&\norm{y_{(0,T)}^{z}(t)-y_{\bbR_0}^{z}(t)}{\clH }<\delta\notag\\
\Longrightarrow&\norm{\fkV(y_{(0,T)}^{z}(t))-\fkV(y_{\bbR_0}^{z}(t))}{\bbR}<\tfrac12\varepsilon
\end{align}
for some~$\delta>0$. In this case, it follows that necessarily~$\norm{y_{(0,T)}^{z}(t)-y_{\bbR_0}^{z}(t)}{\clH }\ge\delta$, for all~$T>T_0$ and all~$t\in[T_0, T]$, otherwise
by~\eqref{contV} it would follow that~$\fkV(y_{(0,T)}^{z}(t))< \fkV(y_{\bbR_0}^{z}(t))+\tfrac12\varepsilon\le\varepsilon$, which would contradict~\eqref{cost2goFI}.
Therefore, we conclude that Assumption~\ref{A:R}  is necessary  in the case our dynamics is autonomous and the value function is continuous.

\subsection{The satisfiability of Assumption~\ref{A:R}}
We have seen that, together with the other assumptions,  Assumption~\ref{A:R} is  sufficient and necessary for the convergence of FTH solutions to IFT ones (at least for the meaningful class of autonomous systems with an associated continuous value function). Thus, in concrete applications we would like to verify the assumption in order to know whether we can expect such convergence. While the verification of the other assumptions is expected to follow, in general,  by standard arguments, the verification of Assumption~\ref{A:R} may require extra work. We shall revisit this point for the concrete examples addressed in the simulations.  

\section{Numerical examples}\label{S:numerics}
We illustrate the derived result. 
 We consider   illustrative examples given by  a nonlinear autonomous {\sc ode} and by a nonautonomous time-periodic linear  {\sc ode}, followed by the {\sc pde} Schl\"ogl model for non-equilibrium phase transitions in chemical reactions~\cite{Schloegl72}.

\subsection{A scalar nonlinear equation}\label{sS:Exscnonlin}
We consider the system
\begin{subequations}\label{ExAn}
\begin{equation}
\dot y= y^{2n-1}+u,\quad y(0)=z\in\bbR,\label{ExAny}
\end{equation}
with~$n\in\bbN_+$, $u=u(t)\in\bbR$,  and the cost
\begin{equation}\label{ExAnJ}
\clJ_{I}(y,u)\coloneqq\frac12\int_I \frac1{2n-1}y(t)^{4n-2}+ u(t)^2\,\rmd t.
\end{equation}
\end{subequations}
In this case, we can find the optimal ITH solutions analytically. The  ITH value function~$\fkV_s(z)=\fkV(z)$ is independent of~$s$, given by
\begin{subequations}\label{VFExAn}
\begin{align}
\fkV(z)&=\tfrac1{2n}wz^{2n},\label{VFExAn-V}
\intertext{where~$w$ solves the Riccati equation}
0&=2w-w^2  +\tfrac1{(2n-1)}.\label{VFExAn-Ric}
\end{align}
\end{subequations}
Indeed, as we  show next, $\fkV$ solves the Hamilton--Jacobi--Bellman equation
\begin{equation}\label{EXAnHJB}
0=\tfrac{\p}{\p t} \fkV +\min_u \{\tfrac{\p}{\p z} \fkV(f_0(y)+u)+J_{I}(y,u)\},
\end{equation}
where
\begin{equation}\notag
f_0(y)\coloneqq y^{2n-1};\quad J_{I}(y,u)\coloneqq\tfrac1{4n-2}y^{4n-2}+\tfrac12 u^2.
\end{equation}
Observe that, denoting~$W(y)\coloneqq\tfrac{\p}{\p y} \fkV(y)=wy^{2n-1}$, by (Karush--Kuhn--Tucker) first order optimality   conditions we can find that an optimal pair~$(y,u)$ satisfies
$u(t)=- W(y(t))$, hence
\begin{align}
\hspace{-1em}&\tfrac{\p}{\p t} \fkV+\min_u \{\tfrac{\p}{\p z} \fkV(f_0+u)+J_{I}\}\notag
\\
\hspace{-1em}&=\tfrac{\p}{\p t} \fkV +W(f_0- W)+\tfrac1{4n-2}y^{4n-2}+\tfrac12 W^2\notag
\\
\hspace{-1em}&=\tfrac{\p}{\p t} \fkV +Wf_0-\tfrac12 W^2 +\tfrac1{4n-2}y^{4n-2}\notag
\\
\hspace{-1em}&=\tfrac{\p}{\p t} \fkV +wy^{4n-2}-\tfrac12w^2 y^{4n-2} +\tfrac1{(4n-2)}y^{4n-2}.\notag
\end{align}
It is well known (as a consequence of the dynamic programming principle) that, for ITH problems and autonomous dynamics, we have that~$\tfrac{\p}{\p t} \fkV=0$, which together with~\eqref{VFExAn-Ric} give us that~$\fkV$ solves~\eqref{EXAnHJB}, since
\begin{align}
 &\tfrac{\p}{\p t} \fkV+\min_u \{\tfrac{\p}{\p z} \fkV(f_0+u)+J_{I}\}\notag\\
&=0 +(2w-w^2  +\tfrac1{(2n-1)})y^{4n-2}=0.\notag
\end{align}

The two solutions of~\eqref{VFExAn-Ric} are
\begin{equation}\notag
w^\pm=\tfrac{1}{-2}\left(-2\pm\sqrt{4+4(2n-1)^{-1}}\right)=1\mp\xi_n
\end{equation}
and since the value function is nonnegative, we find
\begin{equation}\notag
w=w_n\coloneqq  1+\xi_n,\mbox{ with }\xi_n\coloneqq\sqrt{1+(2n-1)^{-1}},
\end{equation}
which leads us to the dynamics~$\dot y= y^{2n-1} + u=  y^{2n-1} - w_ny^{2n-1}$, that is,
the solution~$y=y_\infty$ of the ITH optimal control problem satisfies, for~$t>0$,
\begin{equation}\label{EXAn-optimdyn}
\dot y_\infty=-\xi_n y_\infty^{2n-1}, \quad y_\infty(0)=y_0.
\end{equation}

We compute~$y_\infty$ numerically, for time~$t\in[0,4]$, which we then compare with the computed solutions~$y_{T}$ corresponding to the FTHs for~${T}\in\{1,2,3,4\}$.
We also compare these solutions with the plotted analytic expression~$y_{\rm ex}$ for the solution of~\eqref{EXAn-optimdyn}, namely,
\begin{subequations}\label{poly-yVex}
\begin{equation}\label{poly-yex}
\!\!\!\!y_{\rm ex}(t)\!=\!\begin{cases}
y_0\rme^{-\xi_n t},,\;&\!\!n=1;\\
(y_0^{-2n}+2(n-1)\xi_n t)^{-\frac1{2(n-1)}},\;&\!\!n>1.
\end{cases}
\end{equation}
Further, we compare the value function~\eqref{VFExAn-V},
\begin{equation}\label{poly-Vex}
\fkV(t)=\tfrac{1}{2n}w_ny_0^{2n}=\tfrac{1+\xi_n}{2n}y_0^{2n}
\end{equation}
\end{subequations}
with the computed costs~$\bfJ_{T}=\clJ_{(0,{T})}^0(y_{(0,{T})}^{y_0},u_{(0,{T})}^{y_0})$ for the FTHs problems in~$(0,{T})$.

We present the results of simulations for the case~$n=2$, with~$y_0=1$.
Note that the solution of the free dynamics blows up in finite time, hence we only plot the controlled solutions
in Fig.~\ref{fig:polyA1n2}. We can see that the FTH optimal pairs~$(y_T,u_T)$  approximate the ITH one~$(y_\infty,u_\infty)$. The evolution of the norms of the components of the error~$(y_T-y_\infty,u_T-u_\infty)$ is shown for time $t\in[0,1]$ (i.e., $(s,r)=(0,1)$ in~\eqref{weakconv-sr}). Thus, the simulations agree with the theoretical result.
\begin{figure}[ht]
\centering
\subfigure
{\includegraphics[width=0.45\textwidth]{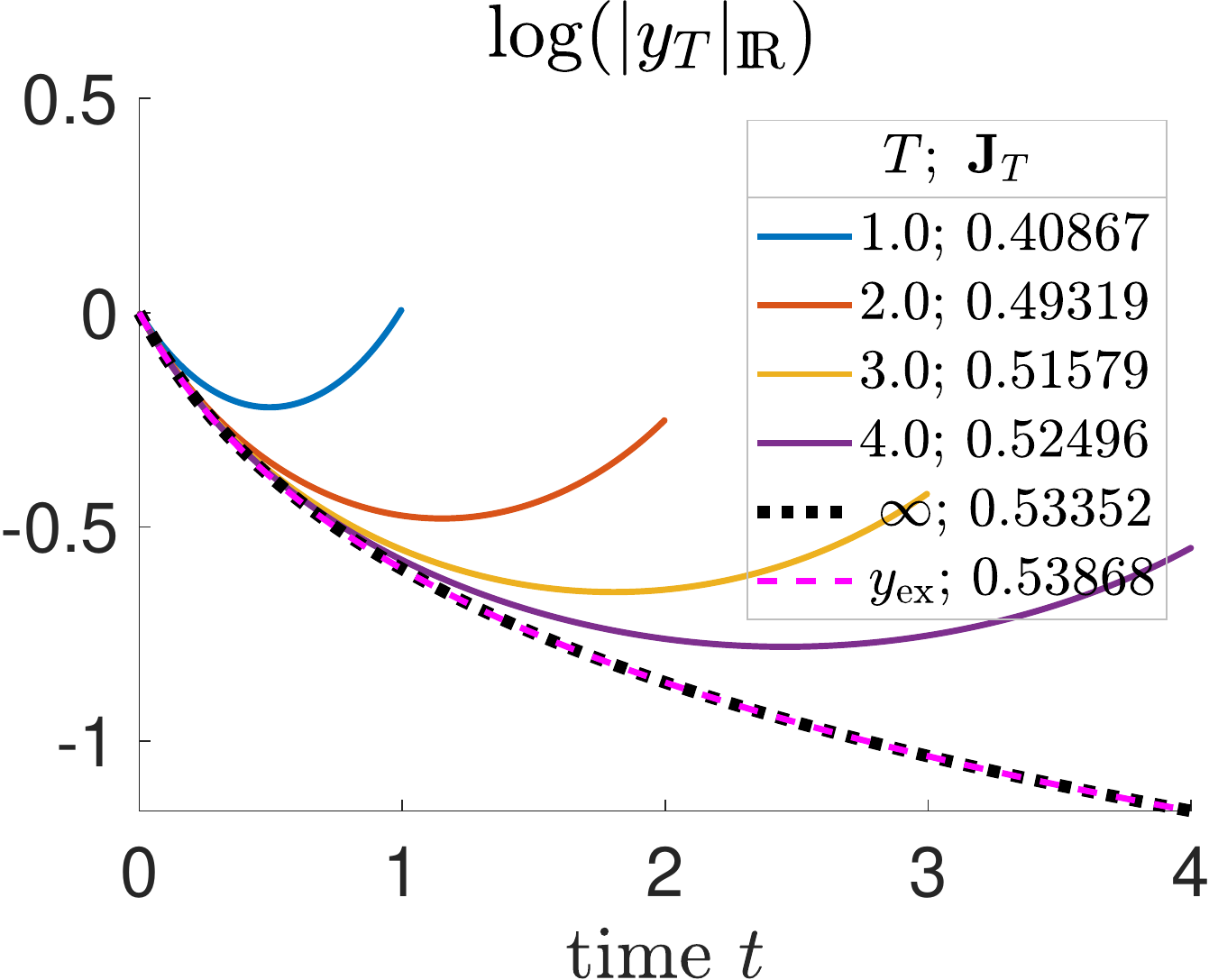}}
\quad
\subfigure
{\includegraphics[width=0.45\textwidth]{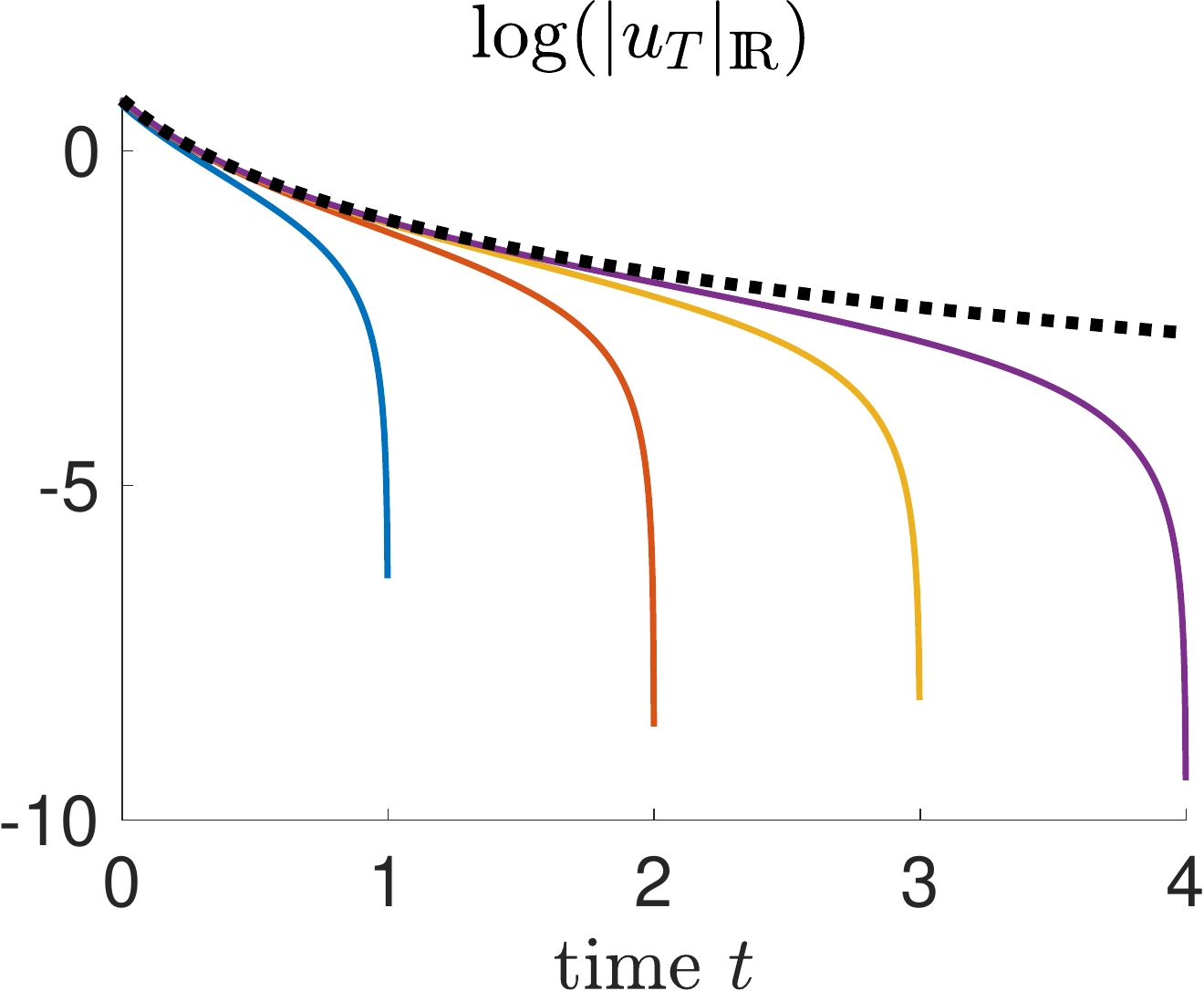}}
\\
\subfigure
{\includegraphics[width=0.45\textwidth]{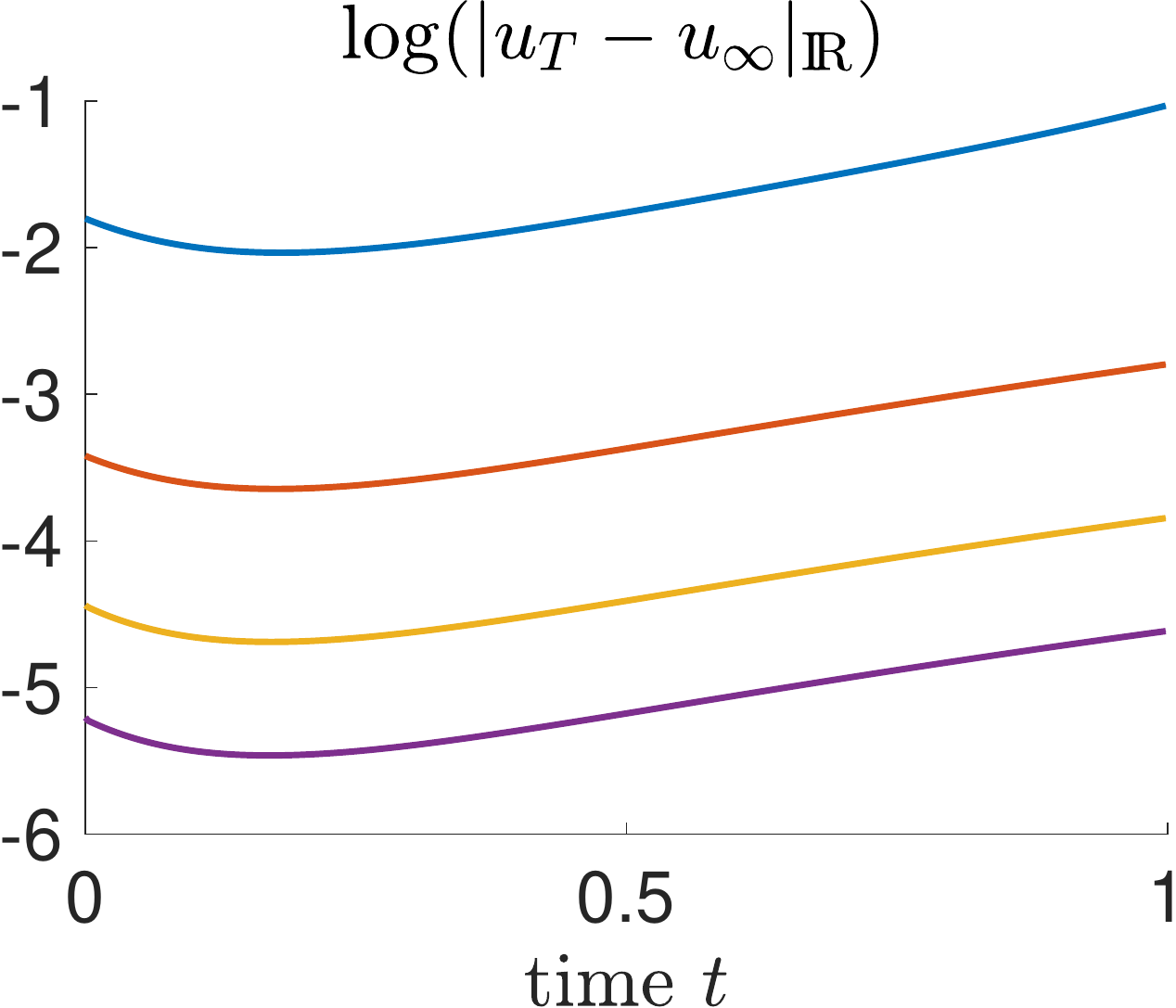}}
\quad
\subfigure
{\includegraphics[width=0.45\textwidth]{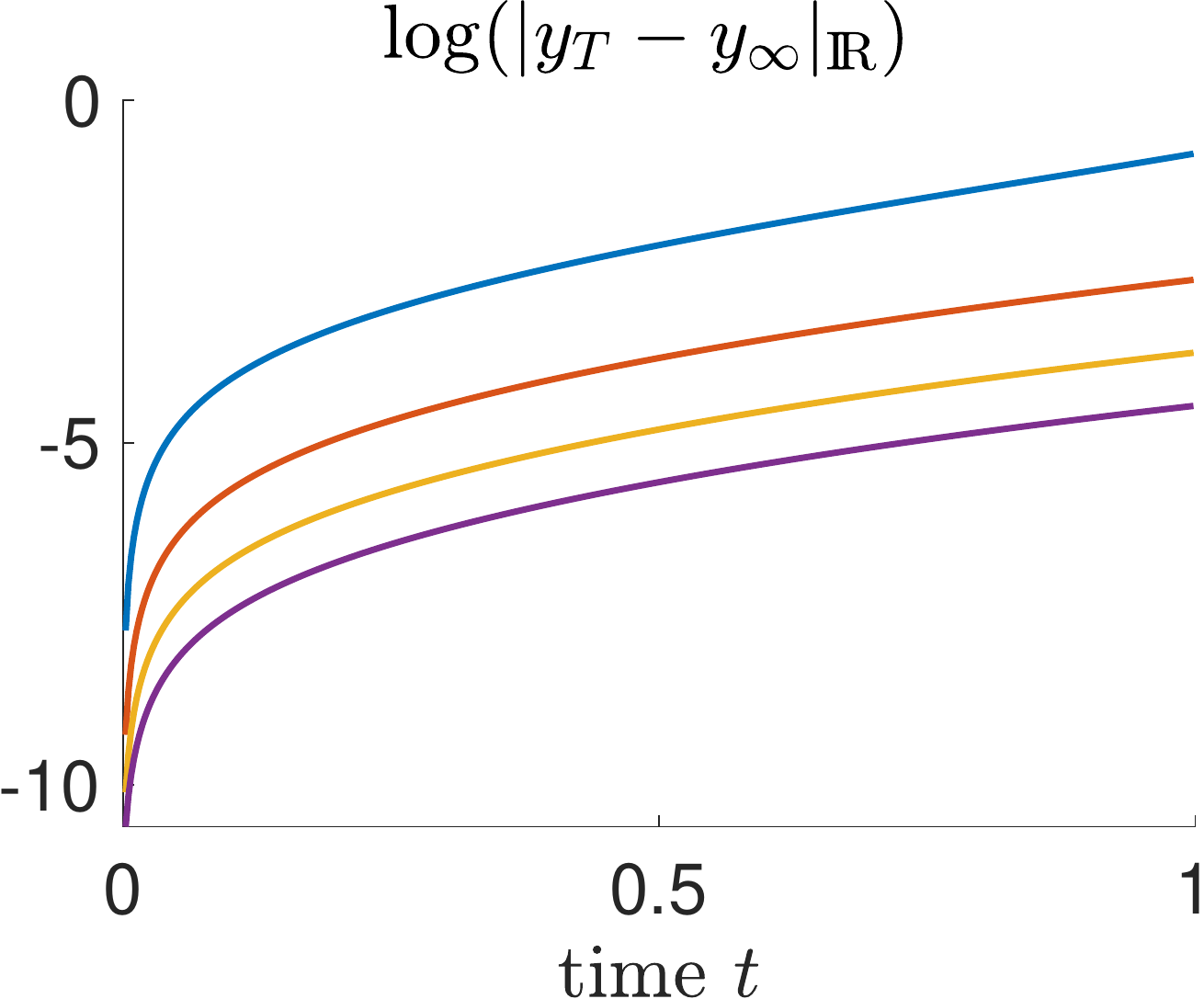}}
\caption{Nonlinear dynamics~\eqref{ExAn}: $n=2$.}
\label{fig:polyA1n2}
\end{figure}

\medskip\noindent
{\bf Remarks on the assumptions}.
Concerning the verification of the assumptions in section~\ref{S:optimalcontrol}, we can take~$\clX^u=L^{2}(I,\bbR)$ and~$\clX^y=W^{1,2}(I,\bbR)\xhookrightarrow{}\clC(\overline I,\bbR)$.
The cost~\eqref{ExAnJ} corresponds to taking $\clP(w)=0$, $\clQ(y)=(2n-1)^{-\frac12}y^{2n-1}$, and~$\clN(u)=u$ in~\eqref{JzP}

We restrict ourselves to the verification of Assumption~\ref{A:R}.
In fact, it is enough to show that our problem can be placed into the particular setting, with~$\norm{\clR(y)}{\bbR}^2\le\fkD_2(\norm{\clQ(y)}{\bbR}^2)$,  discussed in section~\ref{sS:necAR} to conclude that Assumption~\ref{A:R} is satisfied.
To this end, recalling~\eqref{poly-Vex}, we note that~$\norm{\clQ(y)}{\bbR}^2=\frac1{2n-1}y^{4n-2}=\beta_n(\fkV(y))^\frac{4n-2}{2n}$, with~$\beta_n\coloneqq\tfrac{1}{2n-1}(\tfrac{2n}{1+\xi_n})^\frac{4n-2}{2n}$. Now, we can take~$\clR=\sqrt{2\fkV}$ (cf. Assum.~\ref{A:exist-optim}), and obtain $\norm{\clR(y)}{\bbR}^2=\fkD_2(\norm{\clQ(y)}{\bbR}^2)$, with~$\fkD_2(w)\coloneqq2(\tfrac1{\beta_n}w)^\frac{2n}{4n-2}$.

\medskip\noindent
{\bf Remarks on computational details.}
We computed  the optimal solutions~$y_{T}$, for~\eqref{ExAn}, by solving the corresponding first-order optimality system iteratively by following a gradient descent algorithm with Barzilai--Borwein steps~\cite{BarzBorw88}.
In particular, we have followed~\cite{AzmiKun20,AzmiKun21}, where we find  applications to control problems.

\subsection{A time-periodic  linear dynamics}\label{sS:Exper}
We consider the time-periodic dynamics
\begin{subequations}\label{ExPer}
\begin{align}
&\dot y= Ay+B u,\quad y(0)=z\in\bbR^{2\times1},\\
\mbox{with }&A(t)=\varphi(t)\begin{bmatrix}1&1\\1&-3\end{bmatrix},\quad B=\begin{bmatrix}0\\1\end{bmatrix},\\
\mbox{where }&\varphi(t)=1+|\cos(\pi t)|,
\end{align}
with control~$u=u(t)\in\bbR$, 
and minimize, the classical energy cost~$\clJ_{I}^{\clP}=\clJ_{I}^{P}(y,u)$ as
\begin{equation}
\!\!\!\clJ_{I}^{\clP}\!=\tfrac12\norm{y}{L^2(I,\bbR^2)}^2+\tfrac12\norm{u}{L^2(I,\bbR)}^2+\tfrac12\chi \norm{y(\iota_1)}{\bbR^2}^2
\end{equation}
with~$\chi\ge0$.
\end{subequations}
Hence, $\clQ(y)=\norm{y}{\ell^2}$, $\clN(u)=u$, and $\clP(w)=\sqrt{\chi}^\frac12\norm{w}{\ell^2}$.

We  take the initial state~$z=y_0=\begin{bmatrix}1&0\end{bmatrix}^\top$ and present results for the  cases~$\chi\in\{0,1\}$.
We can see that the free dynamics ($u=0$) is unstable. 
Indeed, the eigenvectors of~$A(t)$ are time-independent and given by $e_\pm=(1, -2\pm\sqrt{5})$ with associated time-dependent eigenvalues~$
\alpha_\pm\coloneqq\varphi(t)(-1\pm\sqrt{5})$. Then, writing~$z=z_+e_++z_-e_-$ we find~$z_+\ne0$ and decomposing the solution as~$y(t)=y_+(t)e_++y_-(t)e_-$, we find
\begin{align}
&\dot y_\pm=\alpha_\pm y_\pm,\quad y_\pm(0)=z_\pm,\notag
\end{align}
and~$y_\pm(t)=\rme^{\int_0^t\alpha_\pm(\tau)\,\rmd\tau}z_\pm$.
In particular, $\norm{y(t)}{\bbR^2}\ge\norm{y_+(t)}{\bbR}\ge \rme^{t}\norm{z_+}{\bbR}\to+\infty$.

By the (nonautonomous) Kalman-like rank condition~\cite[Thm.~3]{SilvermanMeadows67}, the controllability at any positive time~$T>0$ follows from the fact that the matrix
$\begin{bmatrix}
B&-A(t)B
\end{bmatrix}$  has full rank (for every~$t>0$). In particular, the ITH optimal control problems are well defined (e.g., by taking a control driving the system to~$y(T)=0$ and switching the control off for~$t>T$).

Note that~$\varphi$ is time-periodic with period~$\varpi=1$. We know that the optimal control is defined, for FTHs (cf.~\cite[Part~III, sects.~1.3 and~1.4]{Zabczyk08}), by the nonnegative solution of the differential Riccati equation
\begin{subequations}\label{Exper-Ricc}
\begin{align}
&\hspace{-2em}\dot \Pi_T+A^\top \Pi_T+ \Pi_TA-\Pi_TBB^\top \Pi_T +\Id=0,\\
&\hspace{-2em}\Pi_T(T)=\chi\Id,\qquad t\in(0,T),
\end{align}
and, for the ITH (cf.~\cite{BarRodShi11,Rod22-eect}), by the nonnegative solution of the periodic Riccati equation
\begin{align}
&\hspace{-2em}\dot \Pi_\infty+A^\top \Pi_\infty+ \Pi_\infty A-\Pi_\infty BB^\top \Pi_\infty +\Id=0,\\
&\hspace{-2em}\Pi_\infty(t)=\Pi_\infty(t+\varpi),\qquad t\in(0,+\infty).
\end{align}
\end{subequations}
Namely, the optimal pairs  satisfy, 
\begin{align}
\hspace{-1em}u_T^z&=-B^\top \Pi_Ty_T^z,&\quad\dot y_T^z&= \left(A-BB^\top \Pi_T\right)y_T^z,
\notag\\
\hspace{-1em}u_\infty^z&=-B^\top \Pi_\infty y_\infty^z,&\quad\dot y_\infty^z&= \left(A-BB^\top \Pi_\infty\right)y_\infty^z,\notag
\end{align}
with~$y_\infty^z(0)=z=y_T^z(0)$.
The results for~$\chi=0$  are shown in Fig.~\ref{fig:periodic_chi0}. We can see that the FTH solutions converge to the ITH one.
\begin{figure}[ht]
\centering
\subfigure
{\includegraphics[width=0.45\textwidth]{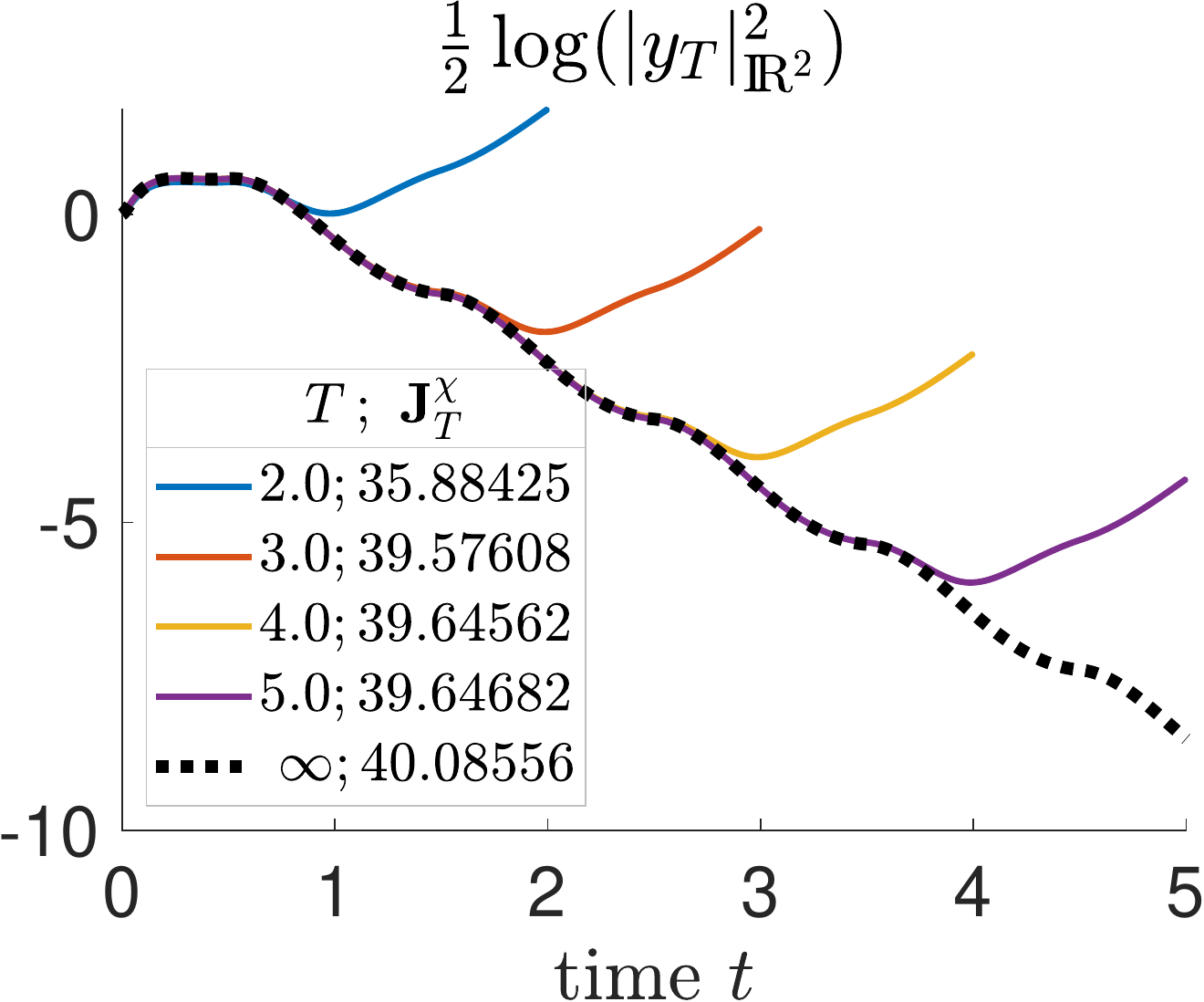}}
\quad
\subfigure
{\includegraphics[width=0.45\textwidth]{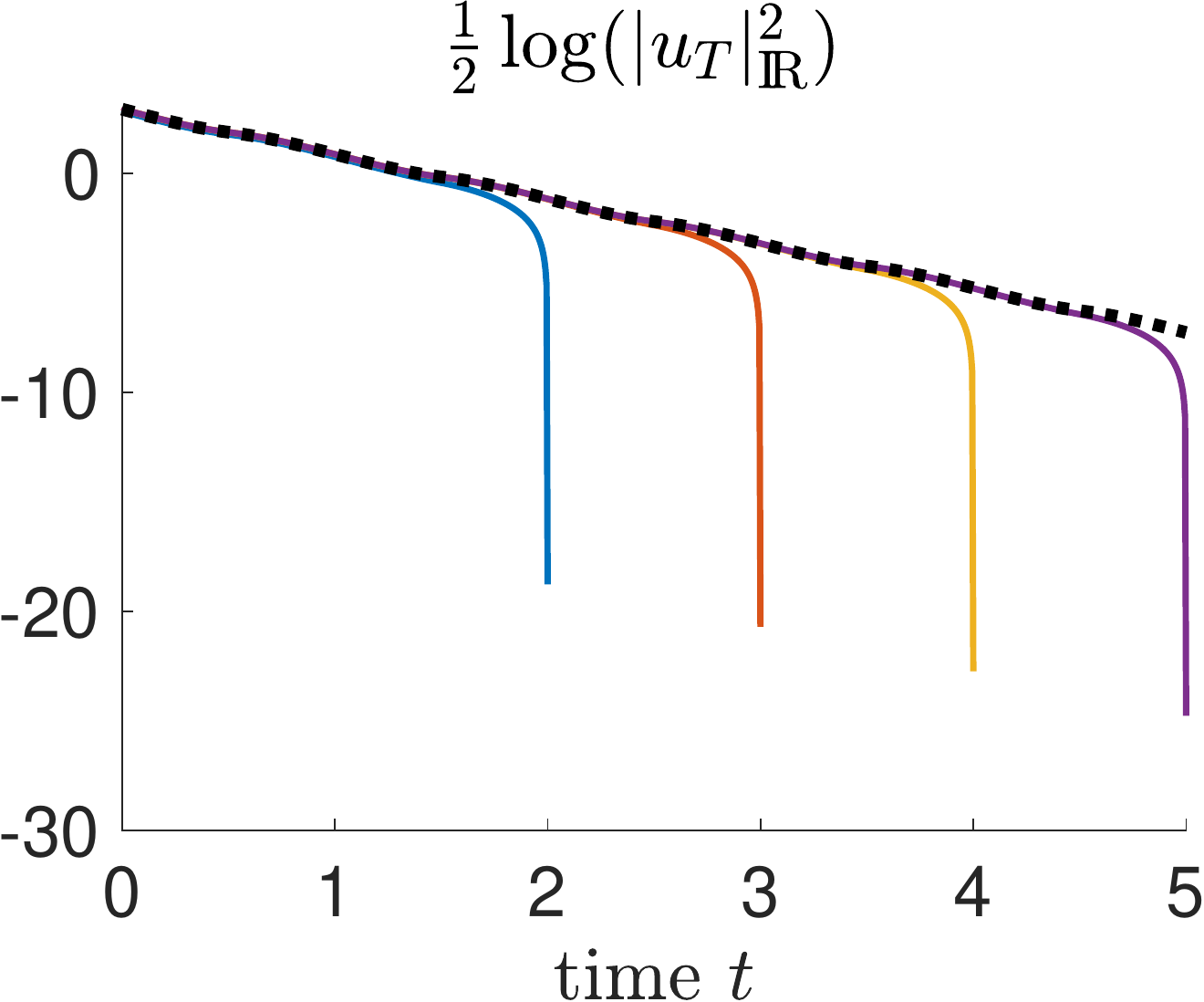}}
\\
\subfigure
{\includegraphics[width=0.45\textwidth]{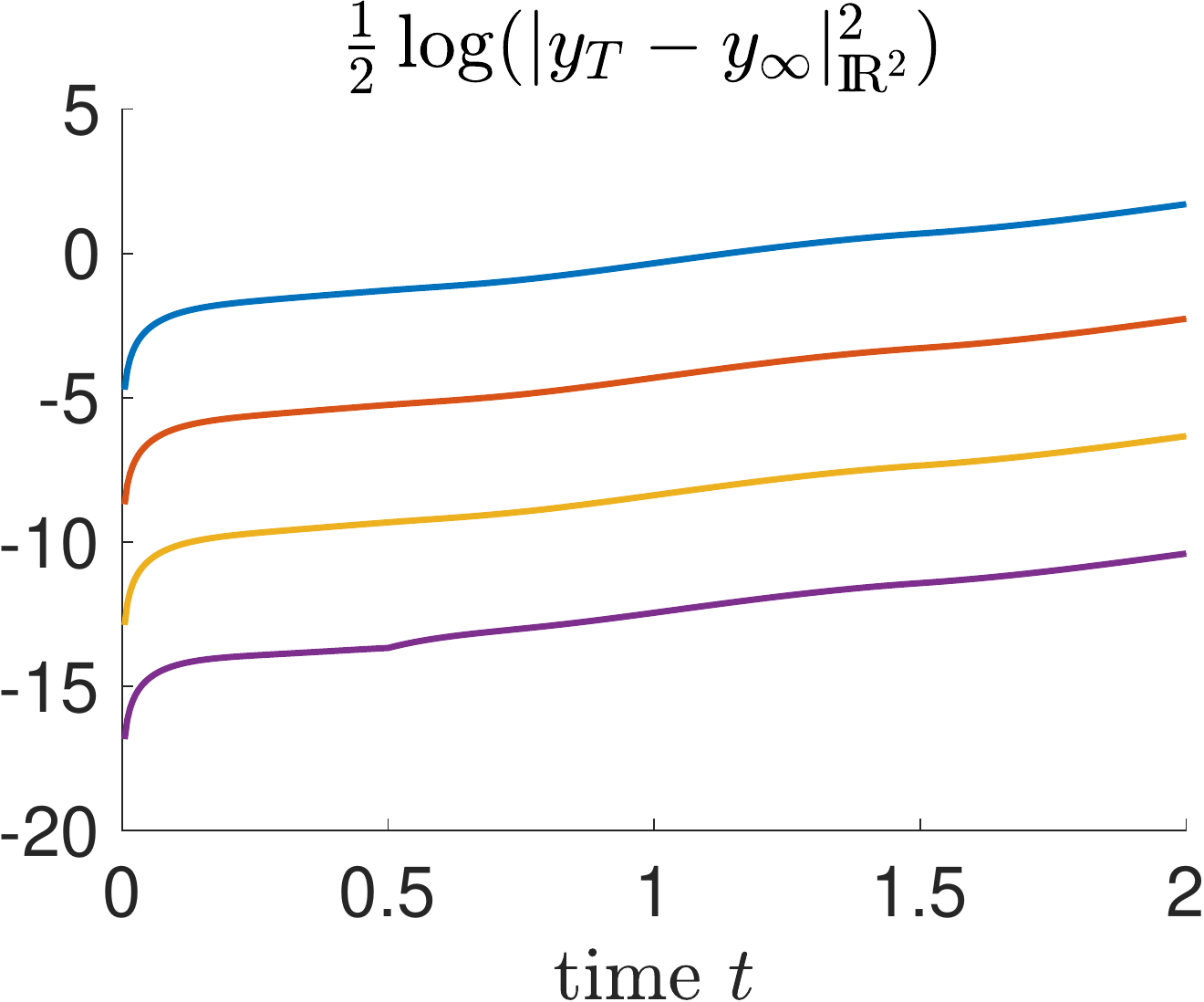}}
\quad
\subfigure
{\includegraphics[width=0.45\textwidth]{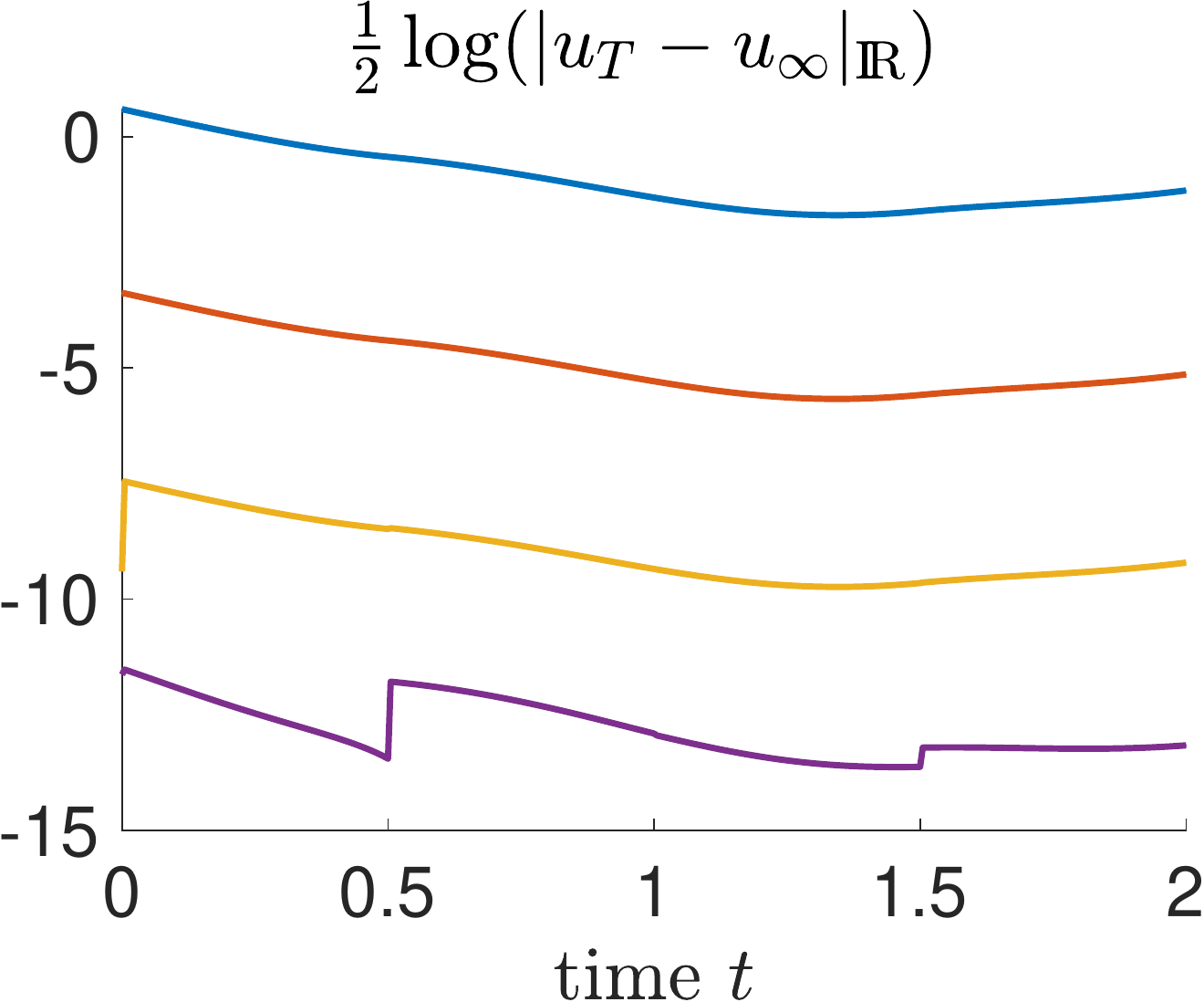}}
\caption{Time-periodic dynamics~\eqref{ExPer}. Case~$\chi=0$.}
\label{fig:periodic_chi0}
\end{figure}
The same convergence is observed in  Fig.~\ref{fig:periodic_chi1} for~$\chi=1$. Further, we observe that the smallness of~$\norm{y(T)}{\ell^2}$ is enhanced for~$\chi=1$.
\begin{figure}[ht]
\centering
\subfigure
{\includegraphics[width=0.45\textwidth]{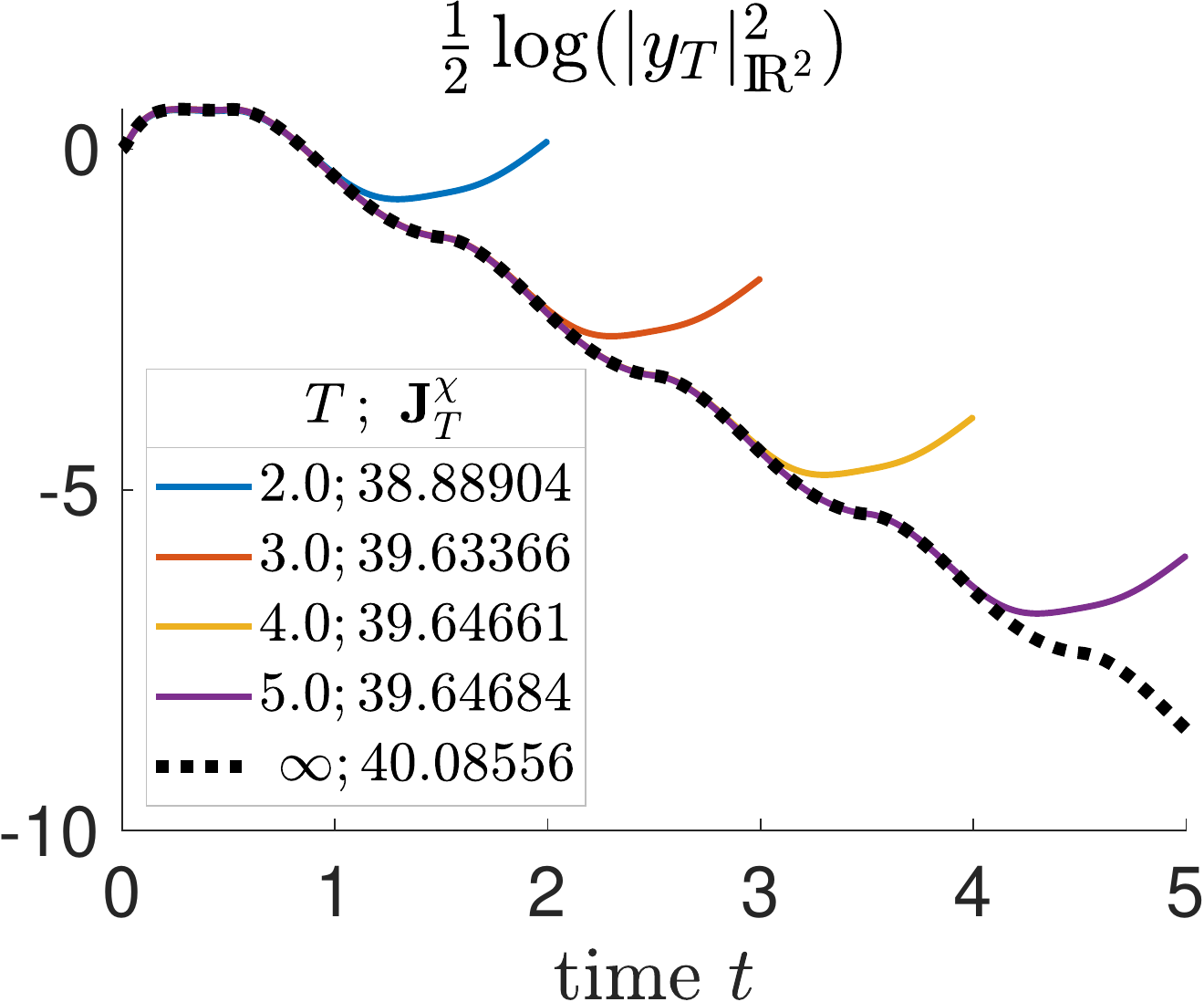}}
\quad
\subfigure
{\includegraphics[width=0.45\textwidth]{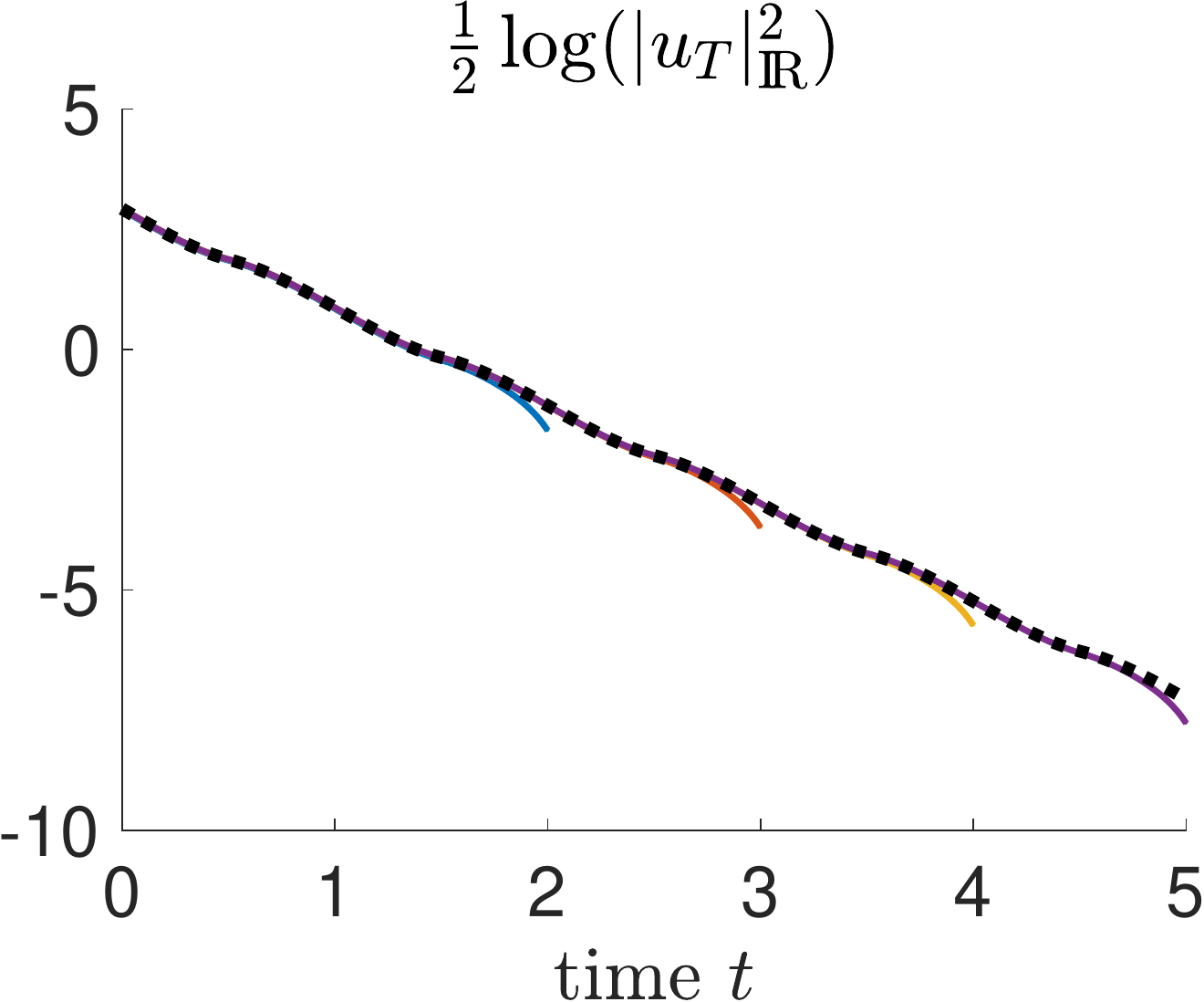}}
\\
\subfigure
{\includegraphics[width=0.45\textwidth]{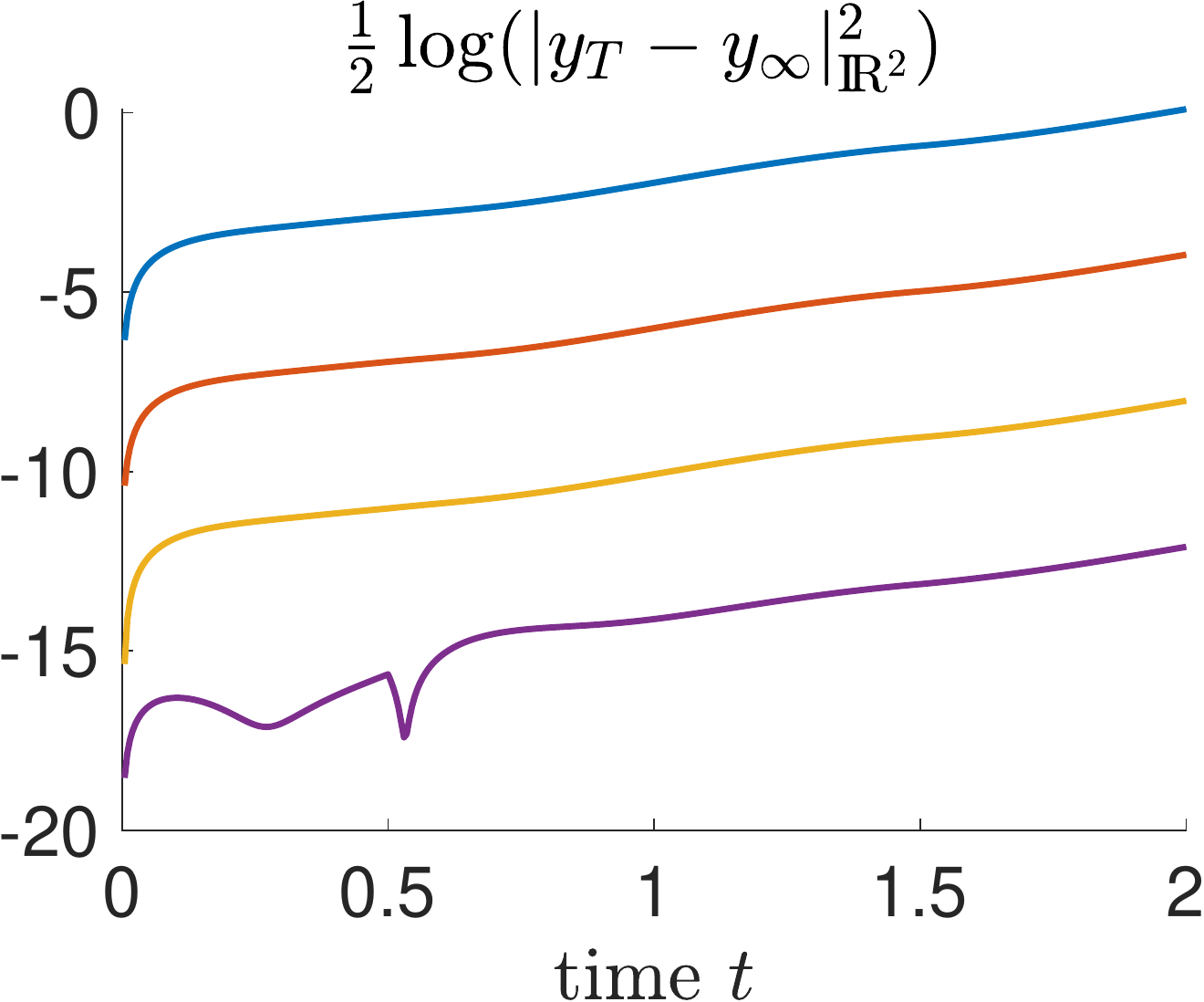}}
\quad
\subfigure
{\includegraphics[width=0.45\textwidth]{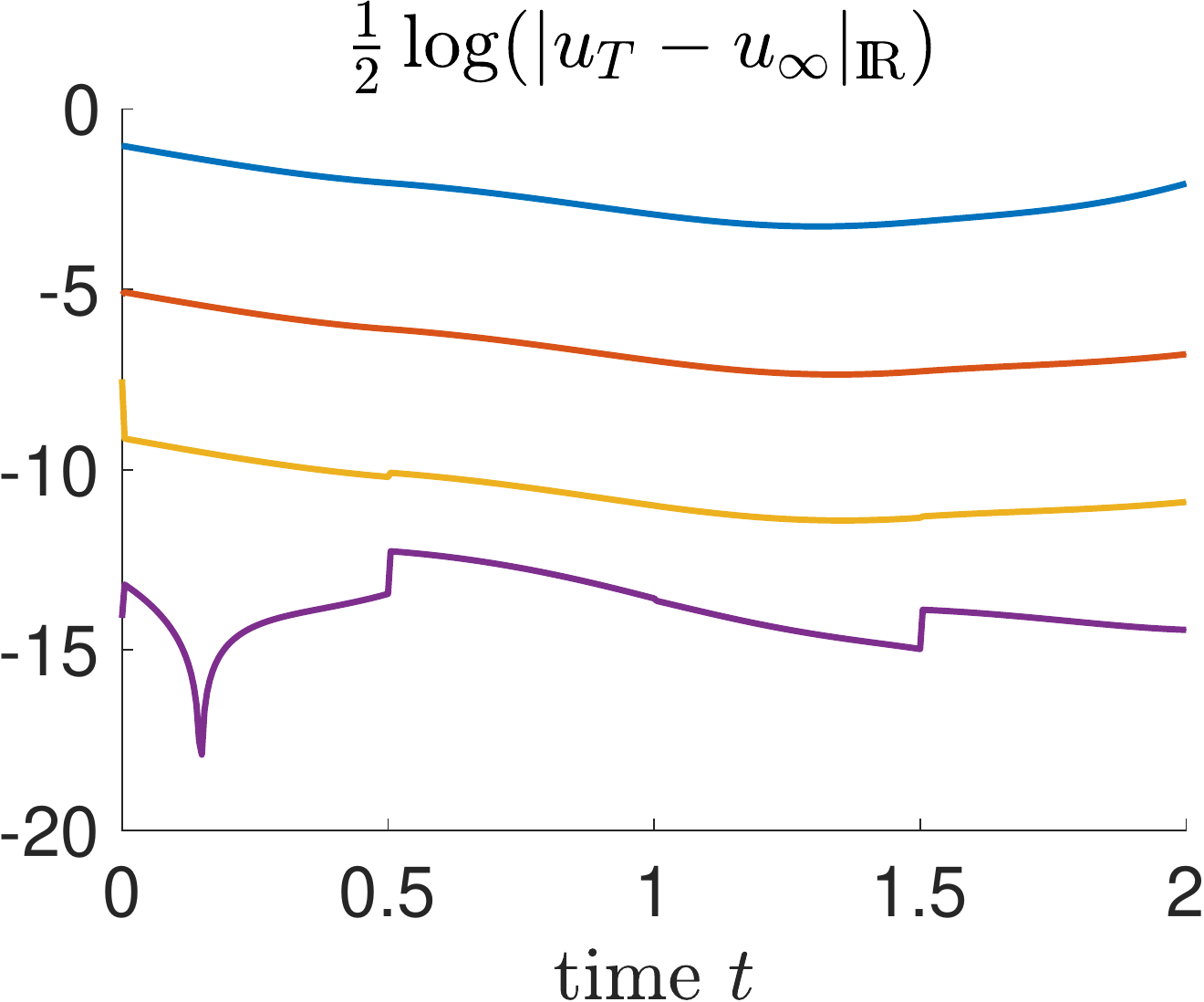}}
\caption{Time-periodic dynamics~\eqref{ExPer}. Case~$\chi=1$.}
\label{fig:periodic_chi1}
\end{figure}

\begin{remark}
For~\eqref{ExPer},  the norm of the state is strictly increasing at time $t=0$, for every control~$u(t)$. Indeed, 
\begin{align}
\tfrac{\rmd}{\rmd t}\norm{y(t)}{\bbR^2}^2
&=2(Ay(t)+Bu(t),y(t))_{\bbR^2}\notag
\end{align}
and,  using~$y_0^\top B=0$ and~$\varphi(0)=2$,
\begin{equation}\notag
\tfrac{\rmd}{\rmd t}\norm{y(t)}{\bbR^2}^2\rest{t=0}=2y_0^\top (A(0)y_0+Bu(0))=4.
\end{equation}
\end{remark}

\medskip\noindent
{\bf Remarks on the assumptions}.
We can choose~$\clX_I^u=L^2(I,\bbR)$ and~$\clX_I^y=W^{1,2}(I,\bbR^2)\subset\clC(\overline I,\bbR^2)$, and verify all the assumptions. Here, we restrict ourselves to Assumption~\ref{A:R}, which we verify  by placing our problem in the particular setting discussed in section~\ref{sS:necAR}, $\norm{\clR(y)}{\bbR}^2\le\fkD_2(\norm{\clQ(y)}{\bbR}^2)$. To this end,  we recall that for matrix linear problems in quadratic costs, the ITH optimal cost is given by~$\tfrac12z(t)^\top\Pi_\infty(t) z(t)$ where $\Pi_\infty$ solves the Riccati equation in~\eqref{Exper-Ricc}. Since~$\varphi$ is continuous and bounded, we can show that there will exist a constant~$C>0$ such that, for all~$s\ge0$ and all~$z\in\bbR^2$, $\fkV_s(z)\le C\norm{z}{\ell^2}^2$. Thus, by setting~$\clR=\sqrt{2\fkV}$ (cf. Assum.~\ref{A:exist-optim}) we arrive at~$\norm{\clR(y)}{\bbR}^2\le2C\norm{\clQ(y)}{\bbR}^2$.

\medskip\noindent
{\bf Remarks on computational details.}
For~\eqref{ExPer}, the optimal solutions were found by solving the  Riccati equations in~\eqref{Exper-Ricc}; following~\cite{Rod22-eect}.

\subsection{The Schl\"ogl parabolic equation}\label{sS:schlogl}
We consider the semilinear Schl\"ogl  model
\begin{subequations}\label{ExSch}
\begin{align}
&\hspace{-1em}\tfrac{\p}{\p t}y=\nu\Delta y -(y-\zeta_1)(y-\zeta_2)(y-\zeta_3)+Bu,\\
&\hspace{-1em}y(0)=y_0\in\clH ,\quad \tfrac{\p}{\p x}y(0,t)=0=\tfrac{\p}{\p x}y(1,t),
\end{align}
in the spatial interval~$\Omega\coloneqq(0,1)$ under homogeneous Neumann boundary conditions, and with
\begin{equation}\label{nuzeta}
\nu=0.1,\quad\zeta=(-1, 0, 2),\quad y_0(x)=\cos(2\pi x^2).
\end{equation}
We take the cost functional as
\begin{equation}\label{J-Schl}
\clJ_{\bbR_0}\coloneqq \tfrac12\gamma\norm{P_{\clE_N} y}{L^2(\bbR_0,{H})}^2+\tfrac12\norm{u}{L^2(\bbR_0,\bbR^{M})}^2,
\end{equation} 
\end{subequations}
where~$P_{\clE_N}\colon H\to\clE_N$ is the orthogonal projection in~$H\coloneqq L^2(\Omega)$ onto the space~$\clE_N$ spanned by the first~$N$ eigenfunctions of the Neumann Laplacian, that is, $\clQ(y)=\gamma^\frac12\norm{P_{\clE_N} y}{H}$, $\clN(u)=\norm{u}{\ell^2}$, and~$\clP=0$. 

The controls are subject to box constraints
\begin{equation}\notag
-C_u\le u_j(t)\le C_u,\quad \mbox{for all}\quad 1\le j\le M,
\end{equation}
for a given constant~$C_u>0$. The existence of stabilizing controls has been shown in~\cite{AzmiKunRod21-arx} for actuators taken as indicator functions~$\indf_{\omega_j}$  of the spatial subdomains
\begin{equation}\notag
\omega_j\coloneqq(c_j-\tfrac{\varrho}{2M},c_j+\tfrac{\varrho}{2M}),\quad c_j\coloneqq\tfrac{2j-1}{2M},\quad 1\le j\le M,
\end{equation}
where~$\varrho\in(0,1)$.  Hence, 
$
Bu\coloneqq{\textstyle\sum\limits_{j=1}^{M}}u_j(t)\indf_{\omega_j}.
$

We present numerical results for the case
\begin{equation}\notag
\gamma= 50,\quad C_u=30,\quad N=20,\quad M=12,\quad \varrho=\tfrac1{10}.
\end{equation}

The initial state and some other time-snapshots of the solution of the free-dynamics ($u=0$) are shown in
Fig.~\ref{fig:Schl-free}.
\begin{figure}[ht]
\centering
\subfigure[Free dynamics.\label{fig:Schl-free}]
{\includegraphics[width=0.45\textwidth]{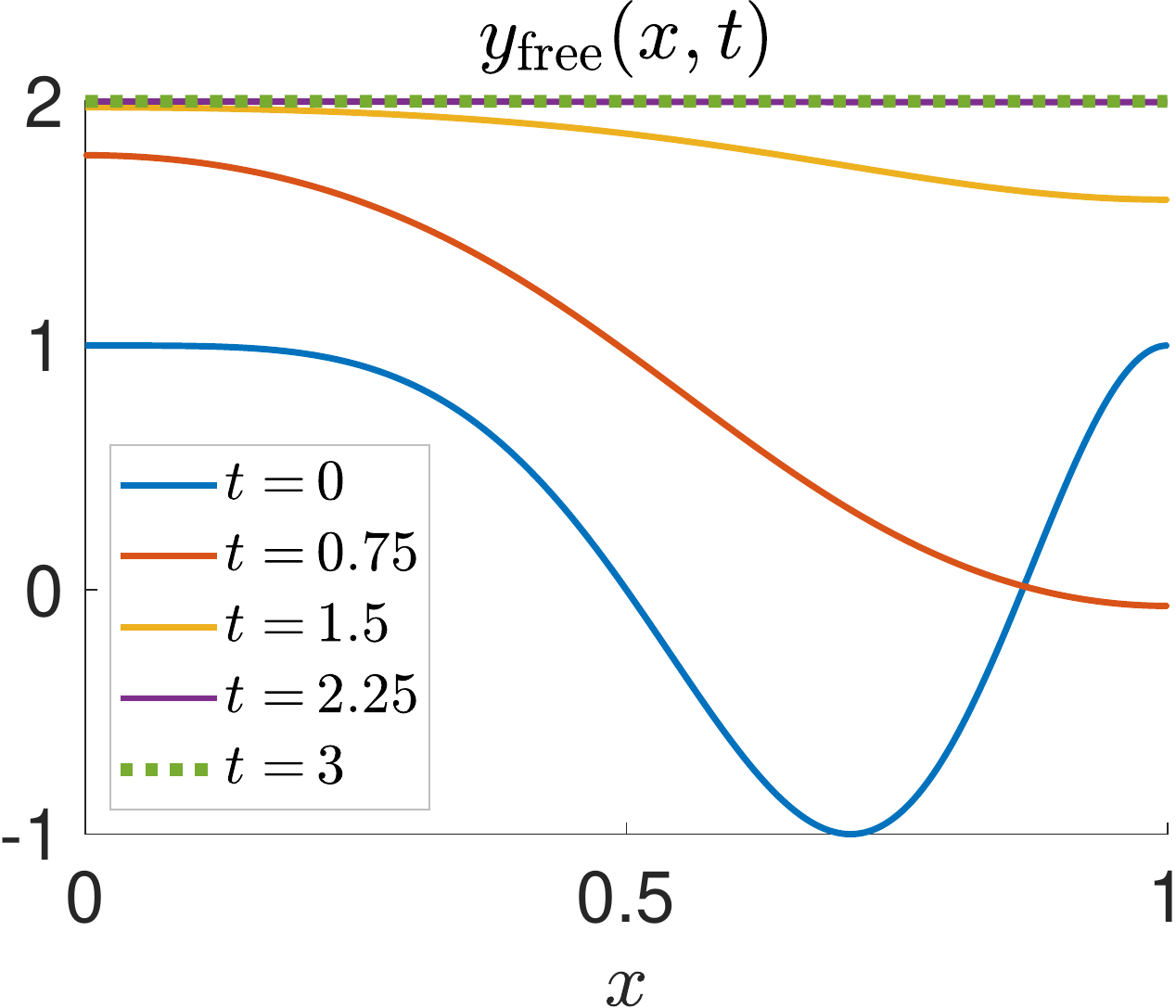}}
\quad
\subfigure[With optimal control.\label{fig:Schl-cont}]
{\includegraphics[width=0.45\textwidth]{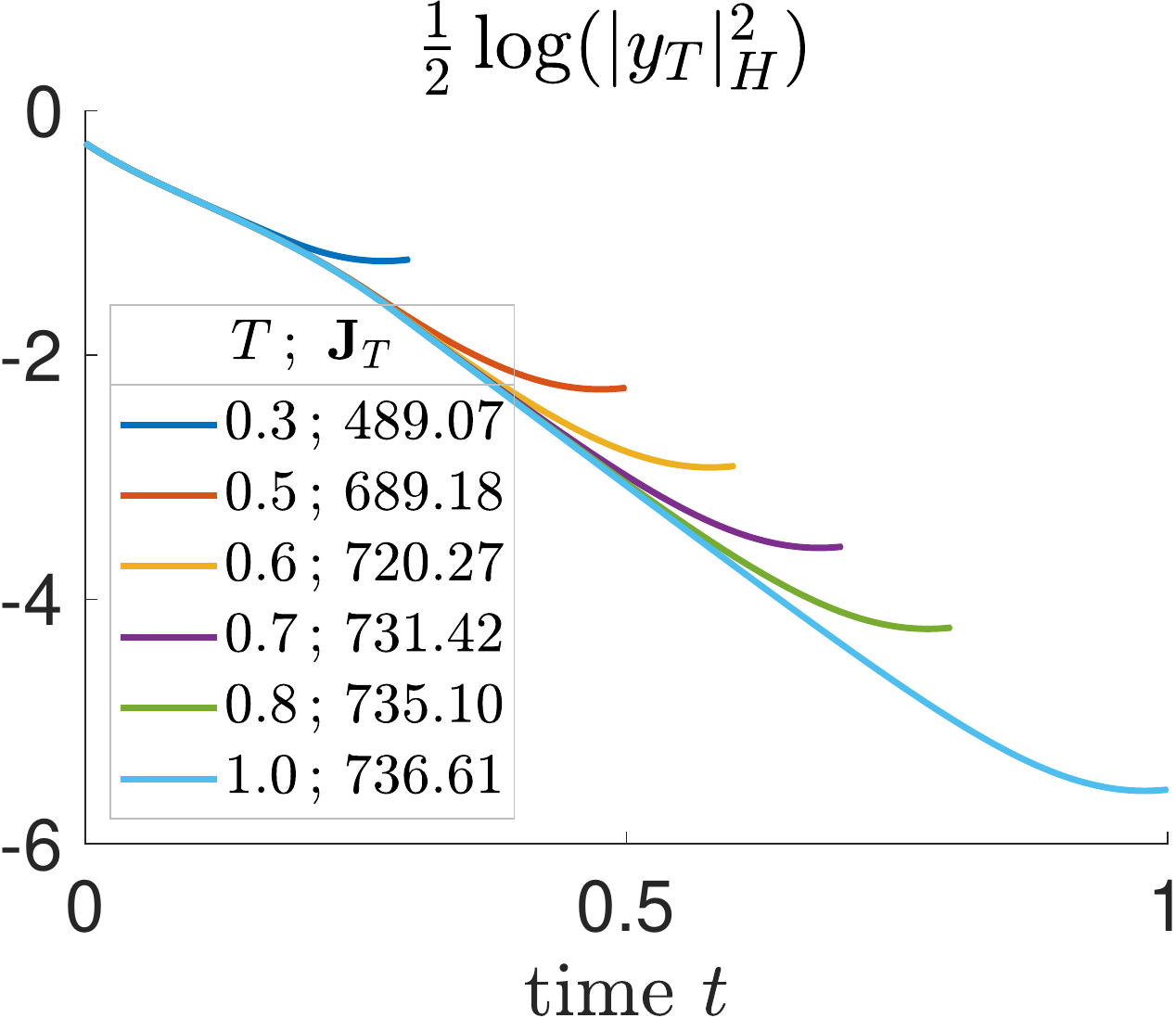}}
\caption{Evolution of the state of~\eqref{ExSch}.}
\label{fig:Schl-state}
\end{figure}
In particular, we see that the state converges to the equilibrium~$y_{\rm eq3}(x,t)\coloneqq\zeta_3=2$.
On the other hand, in Fig.~\ref{fig:Schl-cont} we see that the FTH optimal control solutions, in~$I=(0,{T})$, with ${T}={T}_i$ as
\begin{equation}\notag
({T}_1,{T}_2,{T}_3,{T}_4,{T}_5,{T}_6)\coloneqq(\tfrac{3}{10},\tfrac{5}{10},\tfrac{6}{10}, \tfrac{7}{10},\tfrac{8}{10},1),
\end{equation}
do likely approximate an ITH solution, which converges exponentially to zero, as wanted.

Now, unfortunately, we do not know an optimal solution~$(y_\infty,u_\infty)$ of the ITH, so we cannot compare it with the FTH ones as we did for~\eqref{ExAn} in Fig.~\ref{fig:polyA1n2} and for~\eqref{ExPer} in Figs.~\ref{fig:periodic_chi0} and~\ref{fig:periodic_chi1}. Thus, in order to support the likely ITH limit observed in Fig.~\ref{fig:Schl-cont}, we show the differences in Fig.~\ref{fig:Schl-conv_yu50}, restricted to the time interval~$(0,\tfrac{3}{10})=(0,{T}_1)$ (cf.~\eqref{weakconv-sr}).
\begin{figure}[ht]
\centering
\subfigure
{\includegraphics[width=0.45\textwidth]{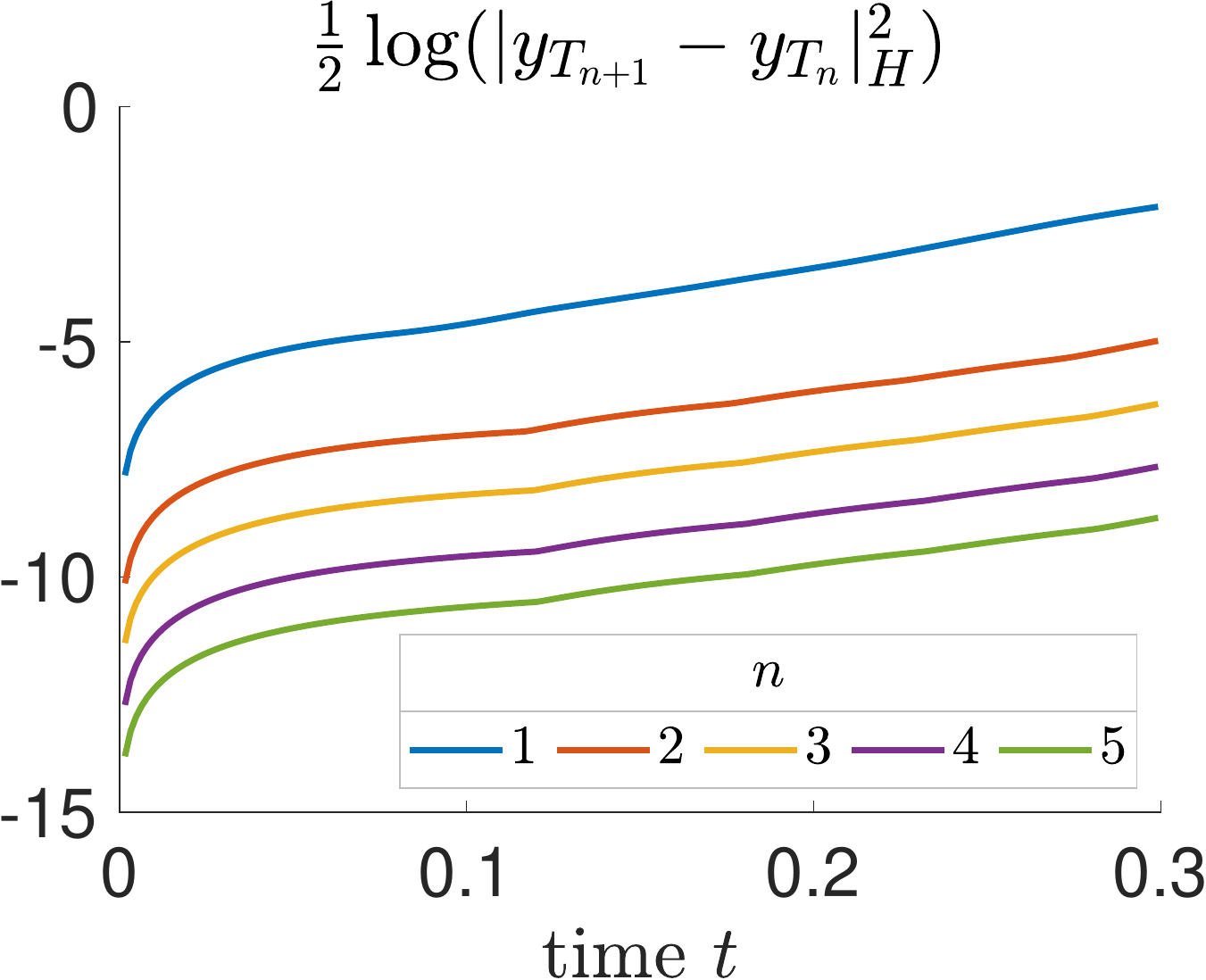}}
\quad
\subfigure
{\includegraphics[width=0.45\textwidth]{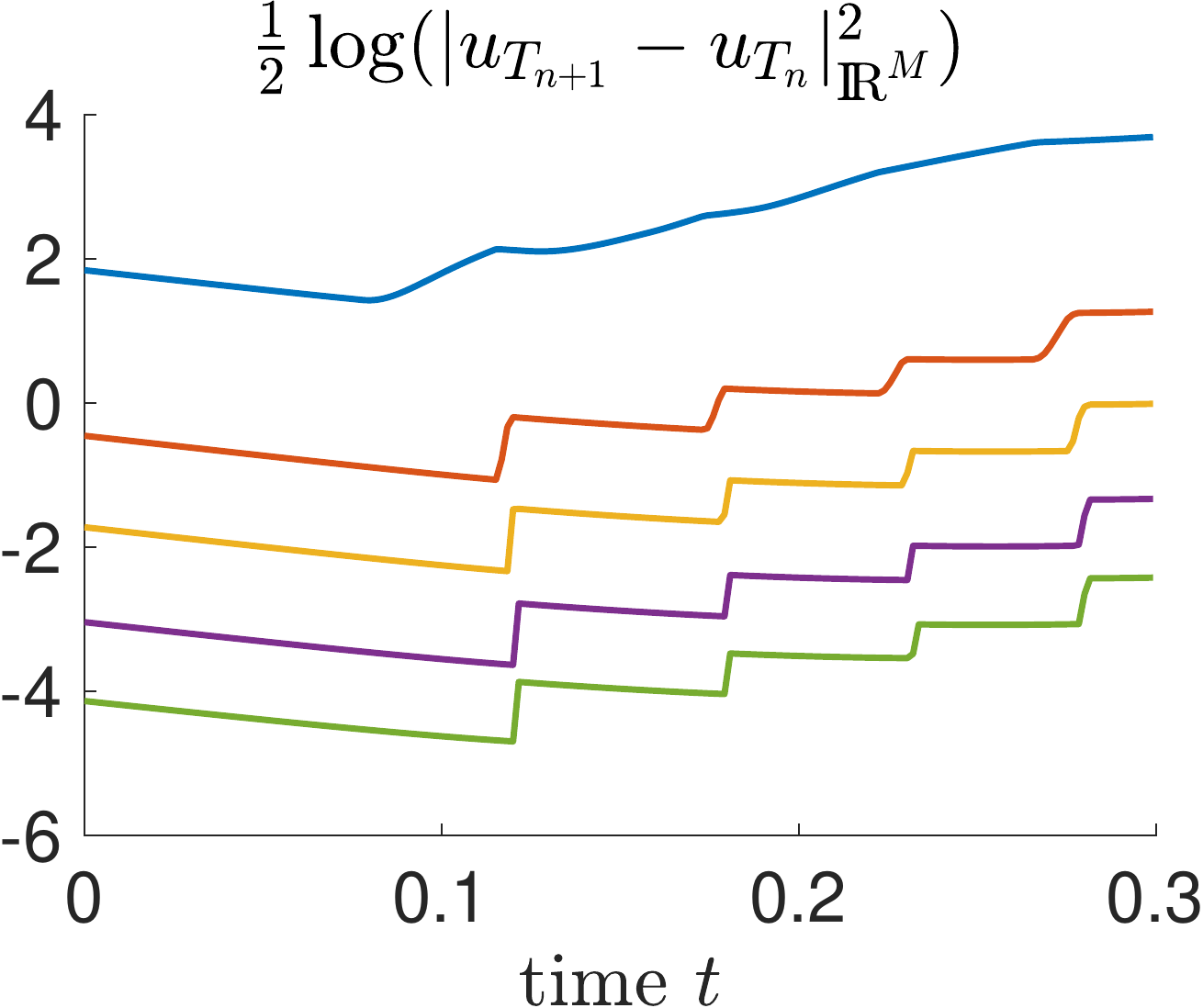}}
\caption{Convergence for~\eqref{ExSch}.}
\label{fig:Schl-conv_yu50}
\end{figure}

Finally, Fig.~\ref{fig:Schl_u50}  shows the evolution of the control coordinates (for the cases~${T}\in\{\tfrac{3}{10},1\}$).
\begin{figure}[ht]
\centering
\subfigure
{\includegraphics[width=0.45\textwidth]{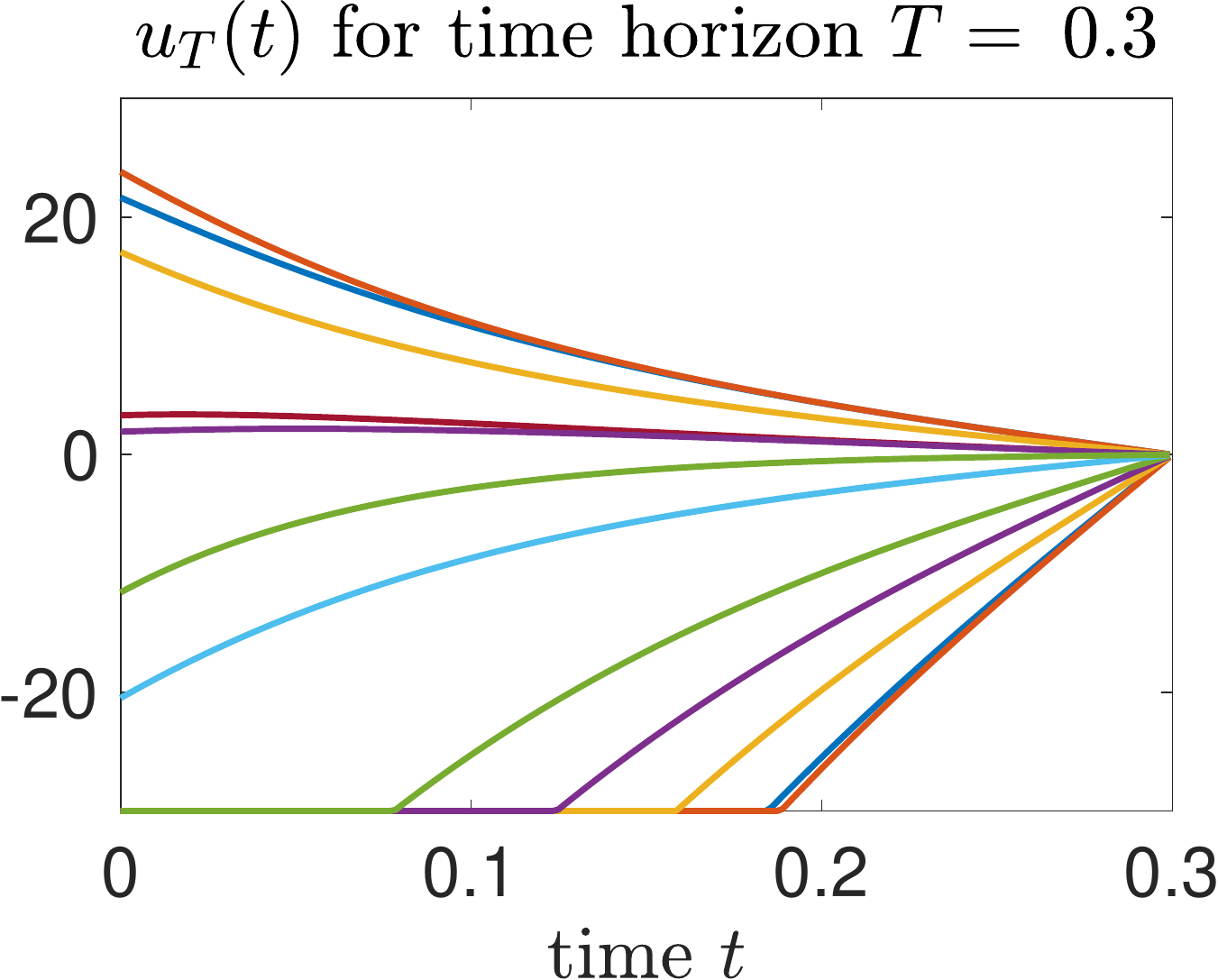}}
\quad
\subfigure
{\includegraphics[width=0.45\textwidth]{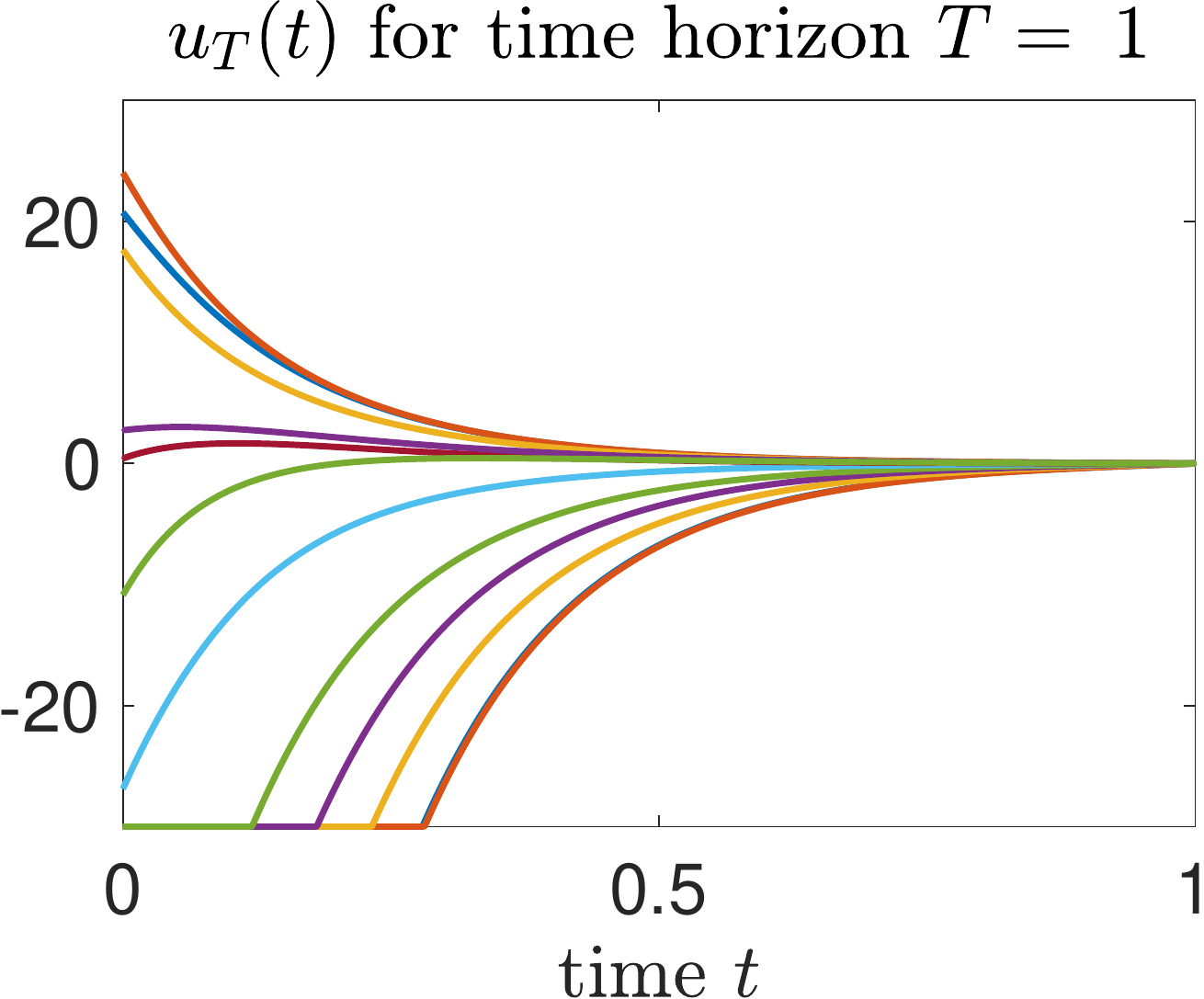}}
\caption{Control input for~\eqref{ExSch}. }
\label{fig:Schl_u50}
\end{figure}

\medskip\noindent
{\bf Remarks on the assumptions}
Recalling the well-posedness of  strong solutions we set~$\clH =H^1(\Omega)$, ~$\clX_I^u=L^2(I,\bbR^M)$ and~$\clX_I^y=\{y\in L^2(I,\rmD(A))\mid \dot y\in L^2(I,H)\}$, where~$\rmD(A)=\{w\in H^2(\Omega)\mid \tfrac{\p w}{\p x}(0)=0=\tfrac{\p w}{\p x}(1)\}$ is the domain of the shifted Neumann Laplacian~$A=-\nu\Delta+\Id$. Again, we can verify all the assumptions. Here, we only address Assumption~\ref{A:R}.
Observe that, now we cannot place our problem in the particular setting discussed in section~\ref{sS:necAR}, $\norm{\clR(y)}{\bbR}^2\le\fkD_2(\norm{\clQ(y)}{\bbR}^2)$, because the optimal cost does not vanish in the kernel~$\clQ^{-1}(\{0\})$ of~$\clQ=\gamma^\frac12\norm{P_{\clE_N} y}{H}$. Indeed, ~$\clQ^{-1}(\{0\})=\clE_N^\perp$ coincides with the orthogonal complement~$\clE_N^\perp$ of~$\clE_N$, in~$H=L^2(\Omega)$. Note that, every~$h\in\clE_N^\perp$ has zero mean, $(h,1)_H=0$, which is a property preserved by the Neumann Laplacian, but  not  by the nonlinearity~$g(y)\coloneqq-y(y+1)(y-2)$, with~$\zeta$ as in~\eqref{nuzeta}. For example, for~$\psi(x)\coloneqq\cos(N\pi x)\in\clE_N^\perp$, we find
\begin{equation}\notag
\xi\coloneqq(g(\psi),1)_H=(-\psi^3+\psi^2+2\psi,1)_H,
\end{equation}
and recalling that~$\psi^2=\frac{1+\cos(2N\pi x)}2$, we obtain~$\xi=\frac12$. Recall that, the constant function~$1$ is in~$\clE_N$.
This means that~$\fkV_s(\psi)\ne0$ and, since~$\clQ(\psi)=0$, we cannot find~$\clR$  satisfying both~$\norm{\clR(w)}{\bbR}^2\le\fkD_2(\norm{\clQ(w)}{\bbR}^2)$ and~$\fkV_s(w)\le \tfrac12\norm{\clR(w)}{\bbR}^2$ (cf. Assum.~\ref{A:exist-optim}).

Therefore, we need another argument to verify Assumption~\ref{A:R}. For this purpose, let $(y_T,u_T)\coloneqq(y_{(s,T)}^z,u_{(s,T)}^z)$ solve Problem~\ref{Pb:OCPI}.
We decompose the state as
$y_T=Y_T+q_T$, with~$q_T=P_{\clE_N}y_T$. We note that
\begin{equation}\notag
(g(y),y)_H=(-y^2+y+2,y^2)_H\le C_1\norm{y}{H}^2
\end{equation}
with~$C_1=\max\{-w^2+w+2\mid w\in\bbR\}$.
Thus, standard energy estimates, will lead us to
\begin{align}
&\hspace{-1em}\tfrac{\rmd}{\rmd t}\norm{y_T}{H}^2\le -\norm{y_T}{\clH}^2 +C_2\norm{y_T}{H}^2+C_2\norm{Bu_T}{H}^2\notag\\
&\hspace{-1em}\le -\norm{Y_T}{\clH}^2 +2C_2\norm{Y_T}{H}^2+2C_2\norm{q_T}{H}^2+C_2\norm{Bu_T}{H}^2\notag\\
&\hspace{-1em}\le -(\alpha_{N+1}-2C_2)\norm{Y_T}{H}^2+2C_2\norm{q_T}{H}^2+C_2\norm{Bu_T}{H}^2\notag
\end{align}
where~$\alpha_{N+1}=\nu N^2\pi^2+1$ is the~$(N+1)$st eigenvalue of~$A$
and~$\norm{y_T}{\clH}^2\coloneqq(Ay_T,y_T)_H$ is a norm equivalent to the usual $H^1(\Omega)$-norm.
 For large~$N$ we have~$\kappa\coloneqq\alpha_{N+1}-2C_1>0$. Then,  integrating over~$I^T=(s,T)$,  \begin{align}
&\norm{y_T(T)}{H}^2+ \kappa\norm{Y_T}{L^2(I^T,H)}^2\notag\\
&\le \norm{z}{H}^2+2C_2\norm{q_T}{L^2(I^T,H)}^2+C_2\norm{Bu_T}{L^2(I^T,H)}^2\notag\\
&\le \norm{z}{H}^2+C_3\fkV_s(z).\notag
\end{align}
By the results in~\cite[Thm.~2.1]{AzmiKunRod21-arx}, for sufficiently large~$M$ and~$C_u$ we have the existence of a  linear feedback operator the saturation of which gives us a feedback control, which stabilizes exponentially our system in the~$H$-norm. Then, necessarily~$\fkV_s(z)\le C_4\norm{z}{H}^2$, and it follows~$\norm{Y_T}{L^2(I^T,H)}^2\le C_5\norm{z}{H}^2$  and~$\norm{y_T}{L^2(I^T,H)}^2\le C_6\norm{z}{H}^2$.
Thus, we must have
$\norm{y_T(T^\circ)}{H}^2\le \frac{2}{T-s}C_6\norm{z}{H}^2$ for some~$T^\circ\in[\frac{T+s}2,T]$. Hence, we can choose a sequence as in Assumption~\ref{A:R}.

\black

\medskip\noindent
{\bf Remarks on computational details}
As in section~\ref{sS:Exscnonlin}, we found the minimizer by solving the first-order optimality system iteratively through Barzilai--Borwein time steps.
To solve the state and adjoint equations from the optimality system, as spatial discretization we used piecewise linear finite elements, and as temporal discretization we used a Crank--Nicolson--Adams--Bashford  scheme.

\section{Concluding remarks}\label{S:finalremks}
We have shown that for a class of systems (possibly unstable and nonlinear) we can approximate an optimal solution of a class of ITH optimal control problems by a limit of a sequence of analogous FTH optimal control problems. This has been validated by numerical simulations performed for both~{\sc ode} and~{\sc pde} models. In section~\ref{sS:Exscnonlin} we considered a nonlinear dynamics for which we could derive the exact analytic expression for the dynamics of the ITH optimal solutions. In section~\ref{sS:Exper} we addressed a time-periodic linear dynamics, and finally, the semilinear Schl\"ogl parabolic model was taken in section~\ref{sS:schlogl} under control box constraints.

The result has been derived under general assumptions, namely, we assume the ITH cost to be finite and several assumptions that can be verified for a large class of concrete systems. The only assumption that will likely be more difficult to check for a given concrete system is Assumption~\ref{A:R}. This assumption is, however, necessary for a meaningful class of problems, as discussed in section~\ref{sS:necAR}, and we have paid particular attention to its verification, for the problems considered in the numerical examples; the arguments in there can be used for other problems as well. Roughly speaking, the function~$\clQ$  in the integral penalization of the state should ``cover the/a most unstable part'' of the dynamics, which reminds us of detectability properties.

\bigskip\noindent
{\bf Acknowledgments.} The author acknowledges partial support from
the Upper Austria Government and the Austrian Science
Fund (FWF): P 33432-NBL.  
He also thanks Behzad Azmi and Karl Kunisch for fruitful discussions on points related with the addressed problem.

\bibliography{Fin2Inf}

\bibliographystyle{plainurl}

\end{document}